\RequirePackage{fix-cm}

\documentclass[smallextended,envcountsect]{svjour3}

\smartqed  % flush right qed marks at end of proof

\usepackage{graphicx}
\usepackage{amsmath,amssymb,mathtools}
\usepackage{algorithm}
\usepackage{algorithmic}

\usepackage{silence}
\usepackage{subcaption}   % loads caption internally
\usepackage{xcolor}
\usepackage[hyphens]{url}

\usepackage{booktabs}
\usepackage{multirow}

\usepackage{ulem}
\newcommand{\best}[1]{\textbf{#1}}
\newcommand{\second}[1]{\uline{#1}}

\graphicspath{{figs/}}

\usepackage{dsfont}
\newcommand{\IRON}{\textsc{IRON}}
\newcommand{\IRONfi}{\textsc{IRON}\(_{\mathrm{FI}}\)}
\newcommand{\R}{\mathbb{R}}
\newcommand{\E}{\mathbb{E}}
\newcommand{\Var}{\mathrm{Var}}
\newcommand{\prox}{\mathrm{prox}}
\newcommand{\Id}{\mathrm{Id}}
\newcommand{\norm}[1]{\left\lVert #1\right\rVert}
\newcommand{\ip}[2]{\left\langle #1,\,#2\right\rangle}
\newcommand{\cO}{\mathcal{O}}
\newcommand{\tr}{\operatorname{tr}}

\spnewtheorem{assumption}[theorem]{Assumption}{\bfseries}{\itshape}
\title{IRON: Implicit Resolvent Optimization under Noise}
\subtitle{Implicit Steps, Smaller Noise: Stationary Variance Control in Stochastic Acceleration}

\author{Valentin Leplat \and Roland Hildebrand}

\institute{
V. Leplat \at Innopolis University
\and
R. Hildebrand \at Moscow Institute of Physics and Technology (MIPT)
}

\date{}

\begin{document}
\maketitle
\begin{abstract}
We study stochastic optimization from a joint continuous--discrete point of view. 
Starting from a second-order stochastic differential equation interpreted as a noisy accelerated gradient flow, we discretize the dynamics by a fully implicit Backward--Euler scheme. This leads to a resolvent, or proximal-type, update, computed in practice through Levenberg--Marquardt, Newton, or trust-region-type inner solves.
The resulting method, denoted by \IRONfi{}, admits a Lyapunov mean-square recursion. 
The main conclusion is that increasing the implicit stepsize \(\alpha\) improves the contraction factor and decreases the stationary mean-square error bound. 
Under sufficiently accurate inner solves, this bound scales as \(O(1/\alpha)\); in particular, for large enough \(\alpha\), the recursion is contractive and the stationary error bound vanishes as \(\alpha\to\infty\).
We establish the theory for smooth strongly convex objectives and provide a sharper quadratic analysis with an explicit stationary constant. 
The numerical experiments support the theory in the strongly convex case and illustrate the same qualitative behavior in nonconvex and learning settings. 
To the best of our knowledge, this fully implicit inertial-resolvent discretization, where the noise acts as an additive perturbation of the resolvent center and yields an \(O(1/\alpha)\) stationary-MSE law, has not been isolated in this form before.
The broader message is that discretization is not only an implementation choice: in stochastic optimization, the numerical integration rule can directly affect the long-time stability of the iterates under noise.

\end{abstract}

% ===================== MAIN TEXT =====================

\section{Introduction}

\paragraph{Motivation.}
Discrete-time analyses dominate stochastic optimization, but the continuous-time viewpoint often clarifies stability and robustness to noise.
We adopt a continuous--discrete perspective that links three ingredients: the underlying accelerated dynamics, the way stochastic perturbations enter the system, and the numerical discretization used to generate an algorithm.
Our main message is that the discretization is not a neutral implementation detail: in the presence of noise, it can shape the stationary spread of the iterates.

Semi-implicit schemes such as NAG-GS~\cite{leplat2023naggssemiimplicitacceleratedrobust} improve robustness compared with fully explicit discretizations, while still retaining explicit components.
Here we take the fully implicit route.
A Backward--Euler discretization of a stochastic accelerated flow yields a resolvent (see Remark \ref{resolventRemark} below), or proximal, update with a natural Levenberg--Marquardt / trust-region inner structure.
Under the noise model considered below, the stochastic perturbation becomes an additive perturbation of the resolvent center.
This leads to a mean-square recursion in which the resolvent contraction improves with the implicit stepsize \(\alpha\), and the stationary mean-square error bound scales as \(O(1/\alpha)\).
Thus, in this model, larger implicit steps reduce the stationary spread induced by noise, provided the inner solves are sufficiently accurate.

\paragraph{Mathematical setup.}
We study stochastic optimization of a smooth objective $f:\R^n\to\R$, in a setting motivated by semi-implicit acceleration (NAG-GS)~\cite{leplat2023naggssemiimplicitacceleratedrobust}.
We model the algorithmic dynamics by the following second-order stochastic differential equation (SDE) in position $x_t\in\R^n$, auxiliary velocity $v_t\in\R^n$, and damping $\gamma_t>0$:
\begin{equation}\label{eq:sde-system}
\left\{
\begin{aligned}
d x_t &= \big(v_t - x_t\big)\,dt,\\
d v_t &= \Big(\frac{\mu}{\gamma_t}\big(x_t - v_t\big) - \frac{1}{\gamma_t}\nabla f(x_t)\Big)\,dt
\;+\; \Sigma^{1/2}\, dW_t,\\
d \gamma_t &= (\mu - \gamma_t)\,dt,
\end{aligned}
\right.
\end{equation}
where \(W_t\) is a standard Brownian motion in \(\R^n\) and \(\Sigma\succeq 0\) is the diffusion covariance.
The case \(\Sigma=\rho^2 I\) corresponds to isotropic noise, with
\[
\E\|\Sigma^{1/2}\eta\|^2=\rho^2 n
\qquad
\text{for } \eta\sim\mathcal N(0,I_n).
\]
Throughout the paper we use the variance proxy
\[
\sigma^2:=\tr(\Sigma),
\]
or any upper bound on \(\E\|\Sigma^{1/2}\eta\|^2\).
The drift terms match the noiseless accelerated flow, while the diffusion term models stochastic perturbations injected in the \(v\)-equation.

\begin{remark} \label{remarkVelocity}
The setup \eqref{eq:sde-system} matches that in~\cite{leplat2023naggssemiimplicitacceleratedrobust}. In this setup $v_t$ does not correspond to the physical velocity. The latter should be rather understood as the difference 
\[
y_t := v_t- x_t.
\]
With this notation the first equation in~\eqref{eq:sde-system} becomes
\[
d x_t = y_t\,dt.
\]
Moreover, 
\[
d y_t = d v_t - d x_t =
\left(
-\left(1+\frac{\mu}{\gamma_t}\right)y_t
-\frac{1}{\gamma_t}\nabla f(x_t)
\right)dt
+
\Sigma^{1/2}dW_t.
\]
Hence $v_t$ should be understood as an auxiliary lifted variable, while $v_t - x_t$ is the physical velocity. In the noiseless drift, the fixed point is given by
\[
\gamma_t=\mu,\qquad v_t=x_t,\qquad \nabla f(x_t)=0.
\]
We keep the variable \(v_t\) because it leads to a convenient
Backward--Euler elimination and matches the lifted-state notation used in
implicit discretizations of accelerated flows.
\end{remark}

In discrete time, \IRONfi{} is obtained by applying Backward--Euler to \eqref{eq:sde-system} and then eliminating the auxiliary variables.
The resulting update has a resolvent form.
This form is central to the paper: it explains the LM/Newton or trust-region inner solve, and it reveals why the injected noise is damped by the implicit resolvent.

\begin{remark}
    \label{resolventRemark}
Throughout the paper, we use the term "resolvent" in the standard operator-theoretic sense. For a proper closed convex function $f$, the Euclidean proximal map satisfies:
\[
\prox_{\lambda f} = (\Id + \lambda \partial f)^{-1}.
\]
In the smooth case considered in most of the paper, this reduces to 
\[
\prox_{\lambda f} = (\Id + \lambda \nabla f)^{-1},
\]
that is, \(x = \prox_{\lambda f}(z)\) if and only if
\[
x + \lambda \nabla f(x) = z.
\]
Thus, in our setting, a resolvent step is equivalently a full implicit gradient step.
\end{remark}

\paragraph{Minimal assumptions.}
\begin{assumption}[Smooth strong convexity]\label{ass:sc}
There exist $0<\mu\le L<\infty$ such that for all $x,y$,
\[
\frac{\mu}{2}\norm{x-y}^2 \le f(x)-f(y)-\ip{\nabla f(y)}{x-y} \le \frac{L}{2}\norm{x-y}^2.
\]
Let $x^\star=\arg\min f$ and $\kappa=L/\mu$.
\end{assumption}

\begin{assumption}[Diffusion noise model]\label{ass:noise}
The stochastic perturbation in \eqref{eq:sde-system} has covariance \(\Sigma\succeq0\).
Over a step of length \(\alpha_k\),
\[
W_{t_{k+1}}-W_{t_k}
=
\sqrt{\alpha_k}\,\eta_k,
\qquad
\eta_k\sim\mathcal N(0,I_n),
\]
so the random vector injected in the \(v\)-update has covariance \(\alpha_k\Sigma\).
We write
\[
\sigma^2:=\tr(\Sigma),
\]
or more generally use \(\sigma^2\) as any upper bound on
\(\E\|\Sigma^{1/2}\eta_k\|^2\).
\end{assumption}

\color{black}
\subsection{Related work.}

\emph{Acceleration and semi-implicit schemes.}
Semi-implicit acceleration (NAG-GS) is derived from an accelerated Nesterov-like SDE together with a semi-implicit Gauss--Seidel discretization, leading to improved robustness and broader stability behavior than fully explicit discretizations~\cite{leplat2023naggssemiimplicitacceleratedrobust}.
In the quadratic setting, the analysis and experiments in~\cite{leplat2023naggssemiimplicitacceleratedrobust} study spectral-radius stability and stationary covariance under stochastic perturbations.
By contrast, our route is fully implicit: each outer step is a resolvent step. Under strong convexity, this resolvent is contractive, and its contraction improves with growing implicit parameter.

\emph{Implicit and resolvent discretizations.}
The fully implicit step is the Euclidean resolvent, or proximal map, of a strongly convex function, a classical object in monotone operator theory~\cite[Ch.~23]{bauschke2017convex}.
In the deterministic setting, the discrete--continuous dictionary of~\cite{wilson2021lyapunov} studies implicit-Euler discretizations of accelerated flows and writes the resulting update in proximal/argmin form, together with a discrete Lyapunov analysis.
See also the ODE-solver viewpoint of~\cite{luo2022from}, which connects acceleration mechanisms to A-stable integrators.
From numerical analysis, Backward Euler is A-stable~\cite{dahlquist1963special,hairer1996sode2}, which is consistent with robustness at large steps.

\emph{Levenberg--Marquardt and trust-region interpretation.}
Solving the fully implicit fixed point by Newton yields linear systems of the form
\[
(I+\lambda\nabla^2 f)\Delta = r,
\]
which have the same algebraic structure as Levenberg--Marquardt (LM) steps and admit a trust-region (TR) interpretation~\cite{levenberg1944method,marquardt1963algorithm,conn2000trust}.
This observation is algorithmically useful: it turns the implicit resolvent step into a standard inner nonlinear solve.

\emph{Stochastic forward--backward and proximal stochastic-gradient methods.}
A nearby but distinct literature is stochastic proximal-gradient or stochastic forward--backward splitting for composite objectives
\[
F(x)=g(x)+r(x),
\]
where one typically performs an explicit stochastic-gradient step on the smooth part \(g\), followed by a proximal step on the regularizer \(r\). Classical examples include FOBOS~\cite{duchi2009forwardbackward}, stochastic forward--backward splitting~\cite{combettes2015sfb}, and variance-reduced proximal stochastic-gradient methods such as Prox-SVRG~\cite{xiao2014proxsvrg}.
In this family, a typical update has the form
\[
x_{k+1}
=
\prox_{\eta_k r}\bigl(x_k-\eta_k \widehat{\nabla} g(x_k)\bigr).
\]
\IRONfi{} is not of this type: it does not apply a regularizer-only proximal step after an explicit stochastic-gradient step. Instead, it evaluates a resolvent of the full objective at a noise-perturbed center.

\emph{Implicit SGD, proximal Robbins--Monro, and stochastic proximal point methods.}
Closer in spirit are implicit stochastic-gradient and stochastic proximal-point methods.
Implicit SGD evaluates the stochastic gradient implicitly at the new iterate and was developed partly to improve stability with respect to stepsize choices~\cite{toulis2017implicit}.
The proximal Robbins--Monro framework interprets stochastic approximation through proximal updates and highlights stability advantages of such implicit steps~\cite{toulis2021proxrm}.
Modern stochastic proximal point methods study updates of the form
\[
x_{k+1}=\prox_{\gamma f_{\xi_k}}(x_k),
\]
as well as variants with minibatching, momentum, inexact inner solves, and variance reduction~\cite{patrascu2018spp,kim2022sppam,milzarek2024spp,traore2024variance,richtarik2024unified}.
Model-based and approximate proximal methods form another related branch~\cite{davis2019modelbased,asi2020aprox}.
These works are important neighbors, but the mechanism studied here is different: \IRONfi{} computes
\[
x_{k+1}
=
\prox_{\lambda_k f}(c_k+\xi_k),
\qquad
c_k = x_k + \frac{v_k - x_k}{1+\tau_k},
\]
with respect to the deterministic full objective \(f\), at an inertial center \(c_k\neq x_k\), and with stochasticity entering as an additive perturbation of the center. {\color{black} Note that the offset $c_k - x_k$ of the center is proportional to the physical velocity $v_k - x_k$ (see Remark \ref{remarkVelocity}). This form makes the inertial structure of the step explicit.}

\emph{Positioning.}
Our goal is not to propose a stochastic optimizer that outperforms state-of-the-art methods such as AdamW, SGD variants, proximal stochastic-gradient methods, or stochastic proximal point methods on standard learning benchmarks.
Rather, we address a specific stability question: how does a fully implicit discretization of an accelerated stochastic flow affect the stationary spread induced by noise?
The distinctive feature of \IRONfi{} is that the stochastic perturbation in the velocity equation becomes an additive perturbation of the resolvent center.
This center-perturbation viewpoint allows us to prove an \(O(1/\alpha)\) stationary mean-square error law under strong convexity, and an explicit stationary constant in the quadratic case.
To the best of our knowledge, this particular fully implicit inertial-resolvent discretization of an accelerated stochastic flow, together with its center-noise mechanism and \(O(1/\alpha)\) stationary-MSE analysis, has not been isolated before.
\color{black}

\color{black}
\subsection{Contributions.}
This paper makes the following contributions.

First, we derive \IRONfi{} from a fully implicit Backward--Euler discretization of the stochastic accelerated flow \eqref{eq:sde-system}. After eliminating the auxiliary variables, the update takes the resolvent form
\[
x_{k+1}=\prox_{\lambda_k f}(c_k+\xi_k),
\]
where the stochastic perturbation injected in the \(v\)-equation appears as an additive perturbation of the resolvent center. This reformulation is the key link between the SDE model, the implicit discretization, and the stationary behavior of the method.

Second, under smooth strong convexity and the diffusion noise model of Assumption~\ref{ass:noise}, we prove a Lyapunov mean-square recursion showing that larger implicit steps strengthen the resolvent contraction and reduce the stationary mean-square error. For constant \(\alpha\) and fixed \(\gamma\), this yields an \(O(\sigma^2/\alpha)\) steady-state bound once \(\alpha\) is sufficiently large; see Theorem~\ref{thm:main}.

Third, we sharpen the analysis in the quadratic case. There, the dynamics decouple in the eigenbasis, leading to an explicit asymptotic stationary constant,
\[
\alpha\,\E\|x_\infty-x^\star\|^2
\;\longrightarrow\;
\frac{\gamma^2}{n}\sigma^2\tr(A^{-2}),
\]
in the isotropic-noise setting. We also give an exact stationary-covariance computation through a discrete Lyapunov equation.

Fourth, we extend the analysis to varying stepsizes, the damping update \(\gamma_k\), and inexact inner solves. In particular, the residual of the inner resolvent equation gives an implementable stopping rule that directly controls the distance to the exact resolvent point and preserves the \(O(1/\alpha)\) stationary scaling.

Finally, we support the theory with numerical experiments designed to isolate the predicted stationary-spread behavior. The quadratic tests compare Monte Carlo simulations with the exact Lyapunov covariance formula. The synthetic strongly convex logistic-regression benchmark supports the predicted \(O(1/\alpha)\) stationary-MSE scaling for a nonlinear objective, with a fitted log--log slope close to \(-1\).
The MNIST softmax-regression experiment illustrates the accuracy, stability, and inner-solve cost of \IRONfi{} on a standard learning task.
\color{black}

\subsection{Notation.}
We use $\ip{\cdot}{\cdot}$ and $\norm{\cdot}$ for the Euclidean inner product and norm.
Constants $c,C,C_i,K_i$ may depend on $(\mu,L,\gamma_0)$ but never on the stepsizes $\{\alpha_k\}$ when we say "independent of $\alpha$".

\subsection{Organization of the paper.}
Section~\ref{sec:scheme} derives the fully implicit scheme \IRONfi{} and its resolvent formulation.
Section~\ref{sec:constant} develops the mean-square analysis, including the constant stepsize theory, the quadratic case, varying stepsizes, inexact inner solves, and mini-batching.
Section~\ref{sec:algorithms} presents pseudocode for the outer resolvent step and the LM/Newton inner solver.
%Section~\ref{sec:experiments} reports numerical experiments that illustrate the contraction and the $O(1/\alpha)$ stationary variance decay, and compares \IRONfi{} to common optimizers. 
Section~\ref{sec:experiments} reports numerical experiments that test the predicted \(O(1/\alpha)\) stationary mean-square scaling, visualize the stationary spread, and assess the practical cost of the inexact resolvent solve.
Finally, Section~\ref{sec:conclusion} summarizes our findings and outlines directions for future work.

%\newpage
% ===================== Scheme =====================
\section{Model - From flow to fully-implicit step (\IRONfi)}\label{sec:scheme}
We now apply Implicit Euler to the SDE system \eqref{eq:sde-system} over a time step of length $\alpha_k>0$.

\subsection*{Step 1: Implicit–Euler discretization of the system}
Using $W_{k+1}-W_k=\sqrt{\alpha_k}\,\eta_k$, $\eta_k\sim\mathcal{N}(0,I_n)$:
\begin{align}
\frac{x_{k+1}-x_k}{\alpha_k} &= v_{k+1} - x_{k+1}, \label{eq:IE-x}\\
\frac{v_{k+1}-v_k}{\alpha_k} &=
   \frac{\mu}{\gamma_k}\big(x_{k+1}-v_{k+1}\big)
   - \frac{1}{\gamma_k}\nabla f(x_{k+1})
   \label{eq:IE-v}\\
&\qquad + \Sigma^{1/2}\,\frac{W_{k+1}-W_k}{\alpha_k} \notag\\
\frac{\gamma_{k+1}-\gamma_k}{\alpha_k} &= \mu - \gamma_{k+1}. \label{eq:IE-gamma}
\end{align}

From \eqref{eq:IE-x}, $v_{k+1}=x_{k+1}+\frac{x_{k+1}-x_k}{\alpha_k}$, and \eqref{eq:IE-gamma} solves explicitly.
For later reference we collect the two state updates:
\begin{equation}\label{eq:v-gamma}
v_{k+1}=\frac{x_{k+1}-x_k}{\alpha_k}+x_{k+1},
\quad
\gamma_{k+1}=\frac{\gamma_k+\alpha_k\mu}{1+\alpha_k}.
\end{equation}

\subsection*{Step 2: Eliminate $v_{k+1}$ and collect terms}
Substitute $v_{k+1}$ into \eqref{eq:IE-v} and multiply by $\alpha_k$:
\begin{equation*}
    \begin{aligned}
        v_{k+1}-v_k \;& =\; -\frac{\mu}{\gamma_k}\big(x_{k+1}-x_k\big)\;-\;\frac{\alpha_k}{\gamma_k}\nabla f(x_{k+1}) \\
        & \;+\;\Sigma^{1/2}\sqrt{\alpha_k}\,\eta_k.
    \end{aligned}
\end{equation*}
But also $v_{k+1}-v_k = \big(x_{k+1}-v_k\big)+\frac{x_{k+1}-x_k}{\alpha_k}$. Rearranging gives
\begin{equation*}
    \begin{aligned}
        & \Big(1+\frac{1}{\alpha_k}+\frac{\mu}{\gamma_k}\Big)\,x_{k+1}\;-\;v_k\;-\;\Big(\frac{1}{\alpha_k}+\frac{\mu}{\gamma_k}\Big)x_k\; \\& +\;\frac{\alpha_k}{\gamma_k}\nabla f(x_{k+1})
\;=\; \Sigma^{1/2}\sqrt{\alpha_k}\,\eta_k.
    \end{aligned}
\end{equation*}

Introduce
\begin{equation}\label{eq:def-params}
\tau_k:=\frac{1}{\alpha_k}+\frac{\mu}{\gamma_k},\quad
c_k:=\frac{v_k+\tau_k x_k}{1+\tau_k},\quad
\lambda_k:=\frac{\alpha_k}{\gamma_k(1+\tau_k)}.
\end{equation}
\iffalse
\color{blue}
\begin{remark}
    (to validate with Roland - rewrite the center in a clearer form)
Equivalently, the center can be written as 
\[
c_k = \frac{v_k + \tau_kx_k}{1+\tau_k} = x_k + \frac{v_k - x_k}{1 + \tau_k}
\]
Thus, $c_k$ is not a convex combination of the position and a physical velocity. Rather, it is the current position shifted in the actual velocity direction $v_k - x_k$, with a damping factor \((1+\tau_k)^{-1}\).
This form makes the inertial structure of the step explicit.

\end{remark}

\color{black}
\fi

Then the previous identity becomes the optimality condition
\begin{equation}\label{eq:opt-cond}
\frac{\alpha_k}{\gamma_k}\nabla f(x_{k+1}) + (1+\tau_k)\Big(x_{k+1}-c_k\Big)
\;=\; \Sigma^{1/2}\sqrt{\alpha_k}\,\eta_k.
\end{equation}

\subsection*{Step 3: Proximal (resolvent) reformulation and the noise term}
Move the right-hand side of \eqref{eq:opt-cond} to the left and divide by $(1+\tau_k)$:
\[
\frac{\alpha_k}{\gamma_k}\nabla f(x_{k+1}) + (1+\tau_k)\Big(x_{k+1}-c_k-\underbrace{\frac{\Sigma^{1/2}\sqrt{\alpha_k}}{1+\tau_k}\eta_k}_{:=\ \xi_k}\Big)=0.
\]
Equivalently,
\begin{equation}\label{eq:prox-noisy-opt}
0 \in \frac{\alpha_k}{\gamma_k}\,\partial f(x_{k+1})+(1+\tau_k)\Big(x_{k+1}-(c_k+\xi_k)\Big),
\end{equation}
whose unique solution (by strong convexity) is the proximal step
\begin{equation}\label{eq:proxsub}
\begin{aligned}
  x_{k+1} & =\arg\min_{x}\Big\{\frac{\alpha_k}{\gamma_k}\,f(x)+\frac{1+\tau_k}{2}\norm{x-(c_k+\xi_k)}^2\Big\} \\
& \;=\; \prox_{\lambda_k f}\!\big(c_k+\xi_k\big),  
\end{aligned}
\end{equation}
with $\lambda_k$ as in \eqref{eq:def-params}. This displays the trust-region center $c_k$ and the resolvent (LM) parameter $\lambda_k$.

%\paragraph{What exactly is $\xi_k$ ? }
\paragraph{Definition and scaling of the center perturbation.}
By construction,
\begin{equation*}
    \begin{aligned}
        & \xi_k \;=\; \frac{\Sigma^{1/2}\sqrt{\alpha_k}}{1+\tau_k}\,\eta_k,\quad 
\eta_k\sim\mathcal{N}(0,I_n), \\
& \E[\eta_k]=0,\ \E[\eta_k\eta_k^\top]=I_n.
    \end{aligned}
\end{equation*}

Therefore $\E[\xi_k]=0$ and
\begin{equation}\label{eq:xi-var}
\begin{aligned}
    \E\norm{\xi_k}^2 \; =\; \frac{\alpha_k}{(1+\tau_k)^2}\,\E\!\left[\eta_k^\top \Sigma \eta_k\right]
\;& =\; \frac{\alpha_k}{(1+\tau_k)^2}\,\mathrm{tr}(\Sigma) \\
\;& \le\; \frac{\alpha_k}{(1+\tau_k)^2}\,\sigma^2,
\end{aligned}
\end{equation}
where we set the variance proxy $\sigma^2:=\mathrm{tr}(\Sigma)$ (or any upper bound on $\E\|\Sigma^{1/2}\eta_k\|^2$). In the isotropic case $\Sigma=\rho^2 I_n$, this equals $\sigma^2=\rho^2 n$.

\paragraph{Summary of the mapping.}
Implicit Euler on \eqref{eq:sde-system} yields \eqref{eq:opt-cond}, the first-order optimality of \eqref{eq:proxsub}. The noise injected in the $v$-equation becomes an additive perturbation of the prox center, $c_k\mapsto c_k+\xi_k$, with second moment bounded as in \eqref{eq:xi-var}. The inner nonlinear-solve structure comes from the Jacobian of the fixed-point map $u\mapsto c_k+\xi_k-\lambda_k\nabla f(u)$, namely $I+\lambda_k \nabla^2 f(u)$.

The perturbation $\xi_k$ is not introduced as an artificial post-processing noise. It is rather the center perturbation obtained after applying the Backward--Euler to the SDE model~\eqref{eq:sde-system} and eliminating the auxiliary variables. Therefore, the theory below analyzes a diffusion-induced "center-noise" model:
\[
x_{k+1} = \prox_{\lambda_k f}(c_k + \xi_k)
\]
This is different from a standard stochastic gradient or mini-batch oracle model, where the randomness enters through an approximation of $\nabla f$ or through a stochastic objective inside the inner solve. Such stochastic-oracle variants are natural in large-scale learning applications, but the main mean-square theory in this paper is proved for the diffusion/center-noise model.

\paragraph{Wilson-Recht-Jordan~\cite{wilson2021lyapunov} Implicit-Euler connection.}
For $\xi_k = 0$ (deterministic setting), Equation~\eqref{eq:proxsub} is exactly the Euclidean prox/resolvent in \cite{wilson2021lyapunov} (Implicit-Euler), with the auxiliary state coupled via \eqref{eq:IE-x} and \eqref{eq:IE-gamma}. See also \cite{luo2022from,hairer1996sode2}.

\paragraph{LM/TR interpretation.}
Optimality of \eqref{eq:proxsub} yields
$0=\frac{\alpha_k}{\gamma_k}\nabla f(x_{k+1})+(1+\tau_k)(x_{k+1}-c_k)$, i.e.,
$x_{k+1}=c_k-\lambda_k\nabla f(x_{k+1})$. Newton on $g(u):=u-(c_k-\lambda_k\nabla f(u))$ uses
$Jg(u)=I+\lambda_k\nabla^2 f(u)$: a Levenberg-Marquardt system. The subproblem is a trust-region model around~$c_k$ \cite{levenberg1944method,marquardt1963algorithm,conn2000trust}.

% ===================== Constant stepsize =====================
\section{Theory}\label{sec:constant}

\subsection{Operator-theoretic background}

We begin with a standard resolvent (proximal) fact: strong convexity makes the proximal map a strict contraction, which will be the main stability mechanism in our stochastic analysis.

\begin{lemma}[Resolvent contraction]\label{lem:resolvent}
If $f$ is $\mu$-strongly convex, then $\prox_{\lambda f}=(\Id+\lambda\nabla f)^{-1}$ is $(1+\lambda\mu)^{-1}$-Lipschitz:
$\norm{\prox_{\lambda f}(u)-\prox_{\lambda f}(v)}\le (1+\lambda\mu)^{-1}\norm{u-v}$.
\end{lemma}
\noindent\textbf{Proof.} See Appendix~\ref{app:proof-1}.

\subsection{Constant-stepsize mean-square analysis}

In this subsection we first isolate the constant-stepsize, frozen-damping regime:
\[
\alpha_k\equiv \alpha\ge 1,
\qquad
\gamma_k\equiv \gamma>0.
\]
This setting is useful because it exposes the main resolvent mechanism without additional bookkeeping. It corresponds either to a frozen-\(\gamma\) variant of the scheme, or to the damping update \eqref{eq:v-gamma} initialized at its fixed point when \(\gamma_0=\mu\), in which case \(\gamma_k\equiv\mu\). The general case with updated \(\gamma_k\) and varying stepsizes is treated separately in Section~\ref{sec:varying}. Define
\[
\tau=\frac{1}{\alpha}+\frac{\mu}{\gamma},
\qquad
\lambda=\frac{\alpha}{\gamma(1+\tau)},
\qquad
c_k=\frac{v_k+\tau x_k}{1+\tau},
\qquad
\hat c_k=c_k+\xi_k,
\]
and introduce the energy
\[
\mathcal{E}_k := f(x_k)-f(x^\star)+\frac{\gamma}{2}\|v_k-x^\star\|^2.
\]

We first collect a few elementary inequalities that connect the energy $\mathcal{E}_k$, the auxiliary state $v_k$, and the prox center $c_k$. These bounds isolate the only places where the stepsize $\alpha$ enters the argument and will be combined in Proposition~\ref{prop:onestep} to prove Theorem~\ref{thm:main}.

\begin{lemma}[Basic energy bounds]\label{lem:energy-eq}
There exist $m_1,m_2>0$ (depending only on $\mu,L,\gamma$) such that
$m_1\norm{x_k-x^\star}^2 \le \mathcal{E}_k \le m_2(\norm{x_k-x^\star}^2+\norm{v_k-x^\star}^2)$.
\end{lemma}

\begin{proof}
Recall $\mathcal{E}_k := f(x_k)-f(x^\star)+\frac{\gamma}{2}\|v_k-x^\star\|^2$.

\noindent\emph{Lower bound.}
By $\mu$-strong convexity and $\nabla f(x^\star)=0$,
\[
f(x_k)-f(x^\star)\;\ge\;\frac{\mu}{2}\|x_k-x^\star\|^2.
\]
Since $\frac{\gamma}{2}\|v_k-x^\star\|^2\ge 0$, it follows that
\[
\mathcal{E}_k \;\ge\; \frac{\mu}{2}\|x_k-x^\star\|^2,
\]
so we may take $m_1:=\mu/2$.

\noindent\emph{Upper bound.}
By $L$-smoothness and $\nabla f(x^\star)=0$,
\[
f(x_k)-f(x^\star)\;\le\;\frac{L}{2}\|x_k-x^\star\|^2.
\]
Therefore,
\[
\mathcal{E}_k
\le \frac{L}{2}\|x_k-x^\star\|^2+\frac{\gamma}{2}\|v_k-x^\star\|^2
\le m_2\big(\|x_k-x^\star\|^2+\|v_k-x^\star\|^2\big),
\]
with $m_2:=\max\{L/2,\gamma/2\}$.
\end{proof}

\begin{lemma}[State-error domination by the energy]\label{lem:energy-dom}
By strong convexity, $f(x_k)-f^\star\ge (\mu/2)\|x_k-x^\star\|^2$, hence
\[
\|x_k-x^\star\|^2+\|v_k-x^\star\|^2 \;\le\; 2 c_E\,\mathcal E_k,
\qquad c_E:=\max\{2/\mu,\,2/\gamma\}.
\]
\end{lemma}

\begin{proof}
By $\mu$-strong convexity,
\[
f(x_k)-f(x^\star)\;\ge\;\frac{\mu}{2}\|x_k-x^\star\|^2
\quad\Longrightarrow\quad
\|x_k-x^\star\|^2 \;\le\; \frac{2}{\mu}\big(f(x_k)-f(x^\star)\big)
\;\le\; \frac{2}{\mu}\,\mathcal{E}_k.
\]
Moreover, since $\frac{\gamma}{2}\|v_k-x^\star\|^2 \le \mathcal{E}_k$, we have
\[
\|v_k-x^\star\|^2 \;\le\; \frac{2}{\gamma}\,\mathcal{E}_k.
\]
Adding the two inequalities and using $c_E:=\max\{2/\mu,2/\gamma\}$ gives
\[
\|x_k-x^\star\|^2+\|v_k-x^\star\|^2
\le c_E\,\mathcal{E}_k + c_E\,\mathcal{E}_k
\le 2c_E\,\mathcal{E}_k.
\]

\end{proof}

\begin{lemma}[Center coupling]\label{lem:center}
For $c_k=\frac{v_k+\tau x_k}{1+\tau}$,
$\norm{c_k-x^\star}^2 \le \frac{2}{1+\tau}\norm{v_k-x^\star}^2 + \frac{2\tau}{1+\tau}\norm{x_k-x^\star}^2$.
\end{lemma}

\begin{proof}
By definition,
\[
c_k-x^\star=\frac{1}{1+\tau}(v_k-x^\star)+\frac{\tau}{1+\tau}(x_k-x^\star).
\]
Using $\|a+b\|^2\le 2\|a\|^2+2\|b\|^2$ with
$a=\frac{1}{1+\tau}(v_k-x^\star)$ and $b=\frac{\tau}{1+\tau}(x_k-x^\star)$ gives
\[
\|c_k-x^\star\|^2
\le \frac{2}{(1+\tau)^2}\|v_k-x^\star\|^2+\frac{2\tau^2}{(1+\tau)^2}\|x_k-x^\star\|^2.
\]
Finally, since $\frac{1}{(1+\tau)^2}\le \frac{1}{1+\tau}$ and
$\frac{\tau^2}{(1+\tau)^2}\le \frac{\tau}{1+\tau}$ for all $\tau\ge 0$, we obtain
\[
\|c_k-x^\star\|^2 \le \frac{2}{1+\tau}\|v_k-x^\star\|^2+\frac{2\tau}{1+\tau}\|x_k-x^\star\|^2.
\]
\end{proof}

\begin{lemma}[Velocity-position coupling]\label{lem:vx}
For $\alpha\ge 1$, the update \eqref{eq:v-gamma} yields
\[
\|v_{k}-x^\star\|^2 \;\le\; 6\,\|x_{k}-x^\star\|^2 \;+\; 4\,\|x_{k-1}-x^\star\|^2.
\]
Consequently, with $m_2$ from Lemma~\ref{lem:energy-eq}, there exists $\bar m_2>0$ (depending only on $\mu,L,\gamma$) such that
\[
\E\mathcal{E}_k \;\le\; \bar m_2\Big(\E\|x_k-x^\star\|^2+\E\|x_{k-1}-x^\star\|^2\Big).
\]
\end{lemma}
\noindent\textbf{Proof.} See Appendix~\ref{app:proof-2}.

\begin{lemma}[Scaling with $\alpha$: product form]\label{lem:scale}
%Let $\tau=\frac{1}{\alpha}+\frac{\mu}{\gamma}$ and $\lambda=\frac{\alpha}{\gamma(1+\tau)}$ with $\alpha\ge 1$, $\gamma>0$, $\mu>0$. Then there exists $K_1>0$ such that
Let $\tau=\frac{1}{\alpha}+\frac{\mu}{\gamma}$ and
$\lambda=\frac{\alpha}{\gamma(1+\tau)}$ with
$\alpha\ge 1$, $\gamma>0$, $\mu>0$. Then there exists
\(K_1>0\), depending only on \(\mu\) and \(\gamma\) and independent of
\(\alpha\), such that
\[
\frac{1}{(1+\lambda\mu)^2} \;\le\; \frac{K_1}{\alpha^2},
\quad
\frac{1}{(1+\lambda\mu)^2}\cdot \frac{\alpha}{(1+\tau)^2} \;\le\; \frac{K_1}{(1+\mu/\gamma)^2}\cdot \frac{1}{\alpha}.
\]
In particular, the product term is $\cO(1/\alpha)$.
\end{lemma}
\noindent\textbf{Proof.} See Appendix~\ref{app:proof-3}.

\begin{proposition}[One-step mean-square bound]\label{prop:onestep}
Under Assumptions~\ref{ass:sc}-\ref{ass:noise}, for $\alpha\ge 1$, there exist $C_1,C_2>0$ (independent of $\alpha$) such that
\[
\E\norm{x_{k+1}-x^\star}^2 \;\le\; \frac{C_1}{\alpha^2}\,\E \mathcal{E}_k \;+\; \frac{C_2\,\sigma^2}{\alpha}.
\]
\end{proposition}
\noindent\textbf{Proof.} See Appendix~\ref{app:proof-5}.
\color{black}
\paragraph{Interpretation.}
Proposition~\ref{prop:onestep} separates the two effects created by the implicit resolvent. The deterministic part is damped by the squared resolvent contraction factor \((1+\lambda\mu)^{-2}=O(\alpha^{-2})\), while the center-noise variance is of size \(O(\alpha)\) before the resolvent is applied; after multiplication by the same contraction factor, it contributes only \(O(1/\alpha)\). This is the elementary mechanism behind the \(O(1/\alpha)\) stationary mean-square scaling.
\color{black}

\begin{remark}[One-step vs.\ two-step: what recurses?]
Proposition~\ref{prop:onestep} gives a \emph{single-step} bound for $x_{k+1}$ in terms of the current energy $\mathcal E_k$.
However, because the implicit-Euler coupling yields
\[
v_k \;=\; x_k + \frac{x_k-x_{k-1}}{\alpha},
\]
any recursion written purely in terms of $\E\|x_k-x^\star\|^2$ is naturally \emph{two-step} (it involves both $x_k$ and $x_{k-1}$ through $v_k$).
To obtain a genuine \emph{one-step} inequality, we introduce a Lyapunov quantity that augments the state with one lag, namely
\[
V_k := \E\|x_k-x^\star\|^2 + \theta\,\E\|x_{k-1}-x^\star\|^2,
\]
for a suitable $\theta=\Theta(1/\alpha)$.
The next theorem states a one-step recursion for $\{V_k\}$ with both the contraction and the stationary mean-square error bound scaling as $1/\alpha$.
\end{remark}

\color{black}
\begin{theorem}[IRON mean-square contraction and $O(1/\alpha)$ stationary error]\label{thm:main}
Under Assumptions~\ref{ass:sc}--\ref{ass:noise} with constant stepsize \(\alpha\ge 1\) and fixed \(\gamma>0\), there exist constants \(G>0\), \(C>0\), and a choice \(\theta=\Theta(1/\alpha)\), all independent of \(\alpha\) except for the explicit scaling of \(\theta\), such that the Lyapunov quantity
\[
V_k \;:=\; \E\|x_k-x^\star\|^2 \;+\; \theta\,\E\|x_{k-1}-x^\star\|^2
\]
satisfies, for all \(k\ge 1\),
\begin{equation}\label{eq:Vk-recursion}
V_{k+1} \;\le\; \frac{G}{\alpha}\,V_k \;+\; \frac{C\,\sigma^2}{\alpha}.
\end{equation}
Consequently, if \(\alpha>G\) and \(\rho:=G/\alpha\in(0,1)\), then for all \(k\ge 1\),
\[
V_{k} \;\le\; \rho^{k-1}V_1 \;+\; \frac{C\,\sigma^2}{\alpha}\cdot\frac{1-\rho^{k-1}}{1-\rho},
\]
and therefore
\[
\E\|x_k-x^\star\|^2 \;\le\; V_k
\;\le\; \rho^{k-1}V_1 \;+\; \frac{C\,\sigma^2}{\alpha(1-G/\alpha)}.
\]
In particular, if \(\alpha\ge 2G\), then
\[
\limsup_{k\to\infty}\E\|x_k-x^\star\|^2
\;\le\;
\frac{2C}{\alpha}\,\sigma^2.
\]
\end{theorem}
\color{black}

\noindent\textbf{Proof.} See Appendix~\ref{app:proof-6}.
\color{black}
\paragraph{Takeaway.}
For sufficiently large \(\alpha\), the implicit step is not merely stable: it becomes more contractive as \(\alpha\) grows, and the stationary mean-square error bound decreases proportionally to \(1/\alpha\). This behavior is opposite to the usual explicit-step intuition, where increasing the stepsize typically amplifies stochastic fluctuations.
\color{black}

We next refine and extend Theorem~\ref{thm:main} in three directions: (i) the quadratic case, where the dynamics decouple in the eigenbasis and yield sharper constants; (ii) varying stepsizes $(\alpha_k,\gamma_k)$, leading to a nonasymptotic bound with a history-weighted noise term; and (iii) inexact inner solves, where residual-based stopping rules translate into explicit perturbations of the mean-square recursion.

\subsection{Quadratic objectives: exact covariance and asymptotic constant}
\label{sec:quad-exact}
Consider $f(x)=\tfrac12 x^\top A x - b^\top x$ with $A\succeq \mu I$ and minimizer $x^\star=A^{-1}b$. 

Set $\Sigma=\rho^2 I$ for clarity, so the noise is isotropic and the coordinates decouple in the eigenbasis of $A$. For a general covariance $\Sigma$, the $O(1/\alpha)$ variance decay still holds with $\sigma^2=\mathrm{tr}(\Sigma)$ as a proxy, but the per-eigendirection decoupling is no longer exact unless $\Sigma$ is diagonal in the eigenbasis of $A$ (for example if $A\Sigma=\Sigma A$). 

Work in the eigenbasis of $A$: let $A=Q\Lambda Q^\top$, $\Lambda=\mathrm{diag}(a_1,\dots,a_n)$, $a_i\in[\mu,L]$, and define errors $e_k:=Q^\top(x_k-x^\star)$ and $w_k:=Q^\top(v_k-x^\star)$. Recall
\[
\tau=\frac{1}{\alpha}+\frac{\mu}{\gamma},\qquad
\lambda=\frac{\alpha}{\gamma(1+\tau)},\qquad
r_i:=\frac{1}{1+\lambda a_i}.
\]
From the prox/optimality form \eqref{eq:proxsub}-\eqref{eq:prox-noisy-opt} and the state update \eqref{eq:v-gamma}, each coordinate $(e_{k,i},w_{k,i})$ obeys the exact linear stochastic recursion
\begin{equation}\label{eq:quad-2x2}
\begin{aligned}
e_{k+1,i} \;&=\; \frac{r_i}{1+\tau}\,w_{k,i} \;+\; \frac{r_i\tau}{1+\tau}\,e_{k,i} \;+\; r_i\,\xi_{k,i},\\[-1mm]
w_{k+1,i} \;&=\; \Big(\tfrac{1}{\alpha}+1\Big)e_{k+1,i} \;-\; \tfrac{1}{\alpha}e_{k,i}
\end{aligned}
\end{equation}
where the injected center noise satisfies $\E[\xi_{k,i}]=0$ and $\E[\xi_{k,i}^2]=\alpha\rho^2/(1+\tau)^2$ (isotropic case).

%\begin{proposition}[Deterministic contraction and explicit stationary mean-square error bound for quadratics]\label{prop:quad-explicit}
\begin{proposition}[Quadratic contraction and explicit stationary constant]\label{prop:quad-explicit}
Let $\alpha\ge 1$, $\gamma>0$, and $\mu>0$. For each eigen-direction $i$:
\begin{enumerate}
\item (Contraction) There exists $C_i=C_i(\mu,L,\gamma)$ such that
\begin{equation*}
    \begin{aligned}
        \E\!\left[e_{k+1,i}^2 + w_{k+1,i}^2 \,\middle|\, e_{k,i},w_{k,i}\right]
& \;\le\; \frac{C_i}{\alpha^2}\,(e_{k,i}^2+w_{k,i}^2) \\
& \quad \;+\; \frac{C_i\,\rho^2}{\alpha}.
    \end{aligned}
\end{equation*}

Consequently, summing over $i$ yields the same $O(1/\alpha^2)$ contraction plus $O(1/\alpha)$ stationary mean-square error bound for the full state $(e_k,w_k)$, and therefore implies the $O(1/\alpha)$ Lyapunov recursion of Theorem~\ref{thm:main} (since $w_k$ couples $(e_k,e_{k-1})$ through \eqref{eq:v-gamma}).

\color{black}
\item (Asymptotic explicit constant) For all sufficiently large \(\alpha\), the linear recursion is stable and admits a stationary law. Denote by \(x_\infty\) the corresponding stationary position variable. Then, as \(\alpha\to\infty\),
\color{black}
\begin{equation*}
    \begin{aligned}
        \lim_{\alpha\to\infty}\; \alpha\cdot \E\|x_\infty-x^\star\|^2
& \;=\; \rho^2\,\gamma^2\,\sum_{i=1}^n \frac{1}{a_i^2} \\
& \;=\; \frac{\gamma^2}{n}\,\sigma^2\,\mathrm{tr}\!\big(A^{-2}\big).
    \end{aligned}
\end{equation*}
Thus the stationary MSE decays exactly like $C_{\mathrm{quad}}/\alpha$ with $C_{\mathrm{quad}}=(\gamma^2/n)\,\sigma^2\,\mathrm{tr}(A^{-2})$.
\end{enumerate}
\end{proposition}

\color{black}

\begin{proof}
See Appendix~\ref{app:quad-proof}.
\end{proof}

\color{black}

\color{black}
\paragraph{Interpretation.}
In the quadratic case the mechanism is completely transparent: in each eigendirection the deterministic transition matrix is \(O(1/\alpha)\), while the effective noise injection into the center has variance \(O(\alpha)\). The resolvent gain \(r_i=O(1/\alpha)\) turns this into a stationary position variance of order \(O(1/\alpha)\), with an explicit constant depending on \(a_i^{-2}\).
\color{black}

\paragraph{Comparison to the general bound.}
Theorem~\ref{thm:main} guarantees $\E\|x_\infty-x^\star\|^2\le \frac{C\,\sigma^2}{\alpha}$ for a constant $C$ depending on $(\mu,L,\gamma)$. In the quadratic case we obtain the sharper, explicit constant

\[
C_{\mathrm{quad}}=\frac{\gamma^2}{n}\,\sigma^2\,\mathrm{tr}(A^{-2})
\;\le\; \frac{\gamma^2}{\mu^2}\,\sigma^2
\quad\text{(since $\mathrm{tr}(A^{-2})\le n/\mu^2$).}
\]

\begin{corollary}[Condition-number bounds for the quadratic constant]
\label{cor:quad-kappa}
For $f(x)=\frac12 x^\top A x-b^\top x$ with $\mu I \preceq A \preceq L I$, the explicit asymptotic constant from Proposition~\ref{prop:quad-explicit} satisfies
\[
C_{\mathrm{quad}}
=\frac{\gamma^2}{n}\,\sigma^2\,\mathrm{tr}(A^{-2})
\;\in\; 
\frac{\gamma^2\sigma^2}{\mu^2}\,\Big[\,\frac{1}{\kappa^2},\ 1\,\Big]
\]
\end{corollary}
\begin{proof}
Since the eigenvalues $a_i$ of $A$ lie in $[\mu,L]$, we have
$\sum_{i=1}^n \frac{1}{L^2} \le \sum_{i=1}^n \frac{1}{a_i^2} \le \sum_{i=1}^n \frac{1}{\mu^2}$.
Multiply by $(\gamma^2/n)\sigma^2$.
\end{proof}

\subsubsection{Exact stationary covariance via a discrete Lyapunov equation}
\label{sec:quad-lyapunov}

Proposition~\ref{prop:quad-explicit} identifies the \emph{asymptotic} stationary noise constant as $\alpha\to\infty$.
For quadratics, one can in fact compute the stationary second moments \emph{exactly} for any fixed stepsize $\alpha$,
because \eqref{eq:quad-2x2} is a linear time-invariant stochastic recursion.
This subsection makes this computation explicit and shows how the constant in Item~2 emerges from the exact formula.

\paragraph{Step 1: write each eigendirection as a linear state-space model.}
Fix an eigenvalue $a=a_i$ and denote $(e_k,w_k):=(e_{k,i},w_{k,i})$.
Introduce $s:=1+\tau$ and define
\begin{equation*}
    \begin{aligned}
        & a_1:=\frac{r\tau}{s},\qquad b_1:=\frac{r}{s},\qquad
c_1:=\Big(1+\frac{1}{\alpha}\Big)a_1-\frac{1}{\alpha},\\ 
& d_1:=\Big(1+\frac{1}{\alpha}\Big)b_1,
    \end{aligned}
\end{equation*}
so that \eqref{eq:quad-2x2} can be written as
\[
\begin{pmatrix} e_{k+1}\\ w_{k+1}\end{pmatrix}
=
\underbrace{\begin{pmatrix} a_1 & b_1\\ c_1 & d_1\end{pmatrix}}_{=:M(a,\alpha)}
\begin{pmatrix} e_{k}\\ w_{k}\end{pmatrix}
+
\underbrace{r\begin{pmatrix}1\\ 1+\tfrac{1}{\alpha}\end{pmatrix}}_{=:g(a,\alpha)}\,\xi_k,
\]
with $\E[\xi_k]=0$ and $\Var(\xi_k)=\frac{\alpha\rho^2}{s^2}$.
Equivalently, with $z_k:=(e_k,w_k)$,
\begin{equation}\label{eq:quad-state}
z_{k+1}=M(a,\alpha)\,z_k + g(a,\alpha)\,\xi_k.
\end{equation}

\paragraph{Step 2: the stationary covariance solves a discrete Lyapunov equation.}
Assume the recursion is stable (in particular, for large $\alpha$, stability holds since $M(a,\alpha)=O(1/\alpha)$).
At stationarity, the mean is zero in error coordinates, and the covariance
\[
P(a,\alpha):=\E[z_k z_k^\top]
=
\begin{pmatrix}
p_{11} & p_{12}\\
p_{12} & p_{22}
\end{pmatrix}
\]
satisfies the \emph{discrete Lyapunov equation}
\begin{equation}\label{eq:quad-lyap}
\begin{aligned}
  & P(a,\alpha)=M(a,\alpha)\,P(a,\alpha)\,M(a,\alpha)^\top + Q(a,\alpha), \\
& Q(a,\alpha):=\Var(\xi_k)\,g(a,\alpha)g(a,\alpha)^\top.  
\end{aligned}
\end{equation}
In particular, the stationary MSE of the coordinate is $\E[e_k^2]=p_{11}$.

\paragraph{Step 3: explicit linear system for the moments.}
Expanding \eqref{eq:quad-lyap} gives a $3\times 3$ linear system for $(p_{11},p_{12},p_{22})$.
Writing $M=\begin{psmallmatrix}a_1&b_1\\c_1&d_1\end{psmallmatrix}$ and
$Q=\begin{psmallmatrix}q_{11}&q_{12}\\q_{12}&q_{22}\end{psmallmatrix}$, one obtains
\[
\begin{pmatrix}
1-a_1^2 & -2a_1b_1 & -b_1^2\\
-a_1c_1 & 1-(a_1d_1+b_1c_1) & -b_1d_1\\
-c_1^2 & -2c_1d_1 & 1-d_1^2
\end{pmatrix}
\begin{pmatrix}p_{11}\\p_{12}\\p_{22}\end{pmatrix}
=
\begin{pmatrix}q_{11}\\q_{12}\\q_{22}\end{pmatrix},
\]
where, using $\Var(\xi_k)=\alpha\rho^2/s^2$ and $g=r(1,1+\tfrac1\alpha)^\top$,
\begin{equation*}
    \begin{aligned}
        & q_{11}=\frac{\alpha\rho^2}{s^2}\,r^2,\qquad
q_{12}=\frac{\alpha\rho^2}{s^2}\,r^2\Big(1+\frac1\alpha\Big),\\
& q_{22}=\frac{\alpha\rho^2}{s^2}\,r^2\Big(1+\frac1\alpha\Big)^2.
    \end{aligned}
\end{equation*}

Thus, for each eigen-direction, the stationary second moments are obtained by solving this explicit $3\times 3$ system.

\paragraph{Step 4: recovering the asymptotic constant.}
As $\alpha\to\infty$, one has $s=1+\tau\to 1+\mu/\gamma$ and
\[
r=\frac{1}{1+\lambda a}
=\frac{1}{1+\frac{\alpha a}{\gamma s}}
\sim \frac{\gamma s}{\alpha a},
\qquad\text{so}\qquad
\frac{r^2\,\alpha\rho^2}{s^2}\sim \frac{\gamma^2\rho^2}{\alpha a^2}.
\]
Moreover, $M(a,\alpha)=O(1/\alpha)$, hence the correction term $M P M^\top$ in \eqref{eq:quad-lyap} is $o(\alpha^{-1})$.
Therefore $p_{11}=\E[e_\infty^2]\sim q_{11}$, yielding
\[
\E[e_{\infty,i}^2]\sim \frac{\gamma^2\rho^2}{\alpha a_i^2},
\]
and summing over $i$ recovers Item~2 of Proposition~\ref{prop:quad-explicit}:
\begin{equation*}
    \begin{aligned}
        \alpha\cdot \E\|x_\infty-x^\star\|^2
\;\to\;
\rho^2\gamma^2\sum_{i=1}^n \frac{1}{a_i^2}
=
\frac{\gamma^2}{n}\,\sigma^2\,\mathrm{tr}(A^{-2}),
    \end{aligned}
\end{equation*}
with $\sigma^2=n\rho^2$. This "exact Lyapunov" viewpoint explains the variance decay mechanism: for large $\alpha$ the linear dynamics matrix $M$ becomes strongly contractive, so the stationary covariance is dominated by the injected noise covariance $Q$, which scales as $1/\alpha$.

%
% ===================== Varying stepsizes =====================

\subsection{Varying stepsizes $\alpha_k$ and $\gamma_k$}\label{sec:varying}
Let $\alpha_k\ge 1$ with $\underline{\alpha}\le \alpha_k\le \overline{\alpha}$, and update $\gamma_k$ via \eqref{eq:v-gamma}.
Then $\gamma_k\in[\min(\gamma_0,\mu),\max(\gamma_0,\mu)]$ for all $k$.

As in the constant stepsize case, the coupling $v_k=x_k+\frac{1}{\alpha_{k-1}}(x_k-x_{k-1})$ makes the natural recursion two-step.
We therefore use a one-lag Lyapunov quantity.

\color{black}
\begin{theorem}[Nonasymptotic bound with varying steps]\label{thm:vary}
Assume \(\alpha_k\ge 1\), \(\underline{\alpha}\le \alpha_k\le \overline{\alpha}\), and let \(\gamma_k\) be updated by \eqref{eq:v-gamma}. There exist constants \(G>0\), \(C>0\), and a choice \(\theta=\Theta(1/\underline{\alpha})\), depending only on \((\mu,L,\underline{\alpha},\overline{\alpha},\gamma_0)\), such that the Lyapunov sequence
\[
V_k \;:=\; \E\|x_k-x^\star\|^2 \;+\; \theta\,\E\|x_{k-1}-x^\star\|^2
\]
satisfies, for all \(k\ge 1\),
\begin{equation}\label{eq:Vk-recursion-varying}
V_{k+1} \;\le\; \frac{G}{\underline{\alpha}}\,V_k \;+\; \frac{C\,\sigma^2}{\alpha_k}.
\end{equation}
Consequently, with \(\rho:=G/\underline{\alpha}\), we have for all \(k\ge 1\),
\[
V_k \;\le\; \rho^{k-1}V_1 \;+\; \sum_{t=1}^{k-1}\rho^{k-1-t}\,\frac{C\,\sigma^2}{\alpha_t}.
\]
In particular, if \(\underline{\alpha}>G\), then
\[
\limsup_{k\to\infty}\E\|x_k-x^\star\|^2
\;\le\;
\limsup_{k\to\infty}V_k
\;\le\;
\frac{C\,\sigma^2}{\underline{\alpha}(1-G/\underline{\alpha})}.
\]
\end{theorem}

\begin{proof}
See Appendix~\ref{app:proof-varying-steps}.
\end{proof}

\paragraph{Interpretation.}
The role of the varying-stepsize result is to show that the \(O(1/\alpha)\) stationary mean-square control mechanism is not an artifact of freezing all parameters. Since the damping update keeps \(\gamma_k\) in a compact interval and the stepsizes remain uniformly bounded below, the resolvent contraction and the noise-scaling estimates hold uniformly in \(k\). The same one-lag Lyapunov argument therefore gives a nonasymptotic bound with a history-weighted noise term.

\color{black}

% ===================== Inexact inner solves =====================

\subsection{Inexact inner solves (LM/Newton)}\label{sec:inexact}
Let $x_{k+1}=\prox_{\lambda_k f}(c_k+\xi_k)$ denote the exact resolvent point, and let the inner solver return an approximation $\tilde x_{k+1}$.

\paragraph{A simple mean-square perturbation bound.}
Assume the inner error satisfies $\E\|\tilde x_{k+1}-x_{k+1}\|^2\le \varepsilon_k^2$.
Then, by $\|a+b\|^2\le 2\|a\|^2+2\|b\|^2$,
\[
\E\|\tilde x_{k+1}-x^\star\|^2
\;\le\;
2\,\E\|x_{k+1}-x^\star\|^2 \;+\; 2\,\varepsilon_k^2.
\]

Therefore, the conclusions of Theorems~\ref{thm:main} and~\ref{thm:vary} persist, up to an additional additive term of order \(\varepsilon_k^2\) in the recursion. The resulting stationary mean-square error bound becomes
\[
O(\sigma^2/\alpha)+O(\sup_k \varepsilon_k^2).
\]
% Thus, preserving the \(O(1/\alpha)\) stationary order only requires the inner-solve contribution to satisfy \(\sup_k\varepsilon_k^2=O(1/\alpha)\). In the implementation and experiments below, we use the stronger tolerance scaling
% \[
% \varepsilon_k^2=O(1/\alpha_k^2),
% \]
% which makes the inexact-solve contribution lower order than the stochastic \(O(\sigma^2/\alpha_k)\) term and therefore negligible in the large-\(\alpha_k\) regime.
Thus, preserving the \(O(1/\alpha)\) stationary order only requires the inner-solve contribution to satisfy \(\sup_k\varepsilon_k^2=O(1/\alpha)\). In the residual-based implementation below, we obtain the stronger effective error scaling
\[
\varepsilon_k^2=O(1/\alpha_k^2),
\]
which makes the inexact-solve contribution lower order than the stochastic \(O(\sigma^2/\alpha_k)\) term and therefore negligible in the large-\(\alpha_k\) regime.

\paragraph{Implementable residual-based stopping.}
Define the fixed-point residual associated with the prox equation by
\[
g_k(u):=u-(c_k+\xi_k)+\lambda_k\nabla f(u),
\quad\text{so that}\quad
g_k(x_{k+1})=0.
\]
Since \(f\) is \(\mu\)-strongly convex, \(g_k\) is
\((1+\lambda_k\mu)\)-strongly monotone. Hence, for any \(u\),
\[
\|u-x_{k+1}\|
\le
\frac{\|g_k(u)\|}{1+\lambda_k\mu}.
\]
Consequently, if the inner solver stops at a point \(\tilde x_{k+1}\) such that
\[
\E\|g_k(\tilde x_{k+1})\|^2\le \delta_k^2,
\]
then
\[
\E\|\tilde x_{k+1}-x_{k+1}\|^2
\le
\frac{\delta_k^2}{(1+\lambda_k\mu)^2}.
\]
Using the same scaling argument as in Lemma~\ref{lem:scale}, uniformly over \(k\) since \(\gamma_k\) remains in a bounded interval, we have
\[
(1+\lambda_k\mu)^{-2}=O(\alpha_k^{-2}).
\]
Therefore, any bounded residual tolerance \(\delta_k=O(1)\) yields
\[
\varepsilon_k^2
:=
\E\|\tilde x_{k+1}-x_{k+1}\|^2
=
O(\alpha_k^{-2}).
\]
Thus, the residual-based criterion automatically satisfies the stronger tolerance scaling used above, and the inexact-solve contribution is negligible compared with the stochastic \(O(\sigma^2/\alpha_k)\) term in the stationary mean-square bound.

% ===================== Mini-batching =====================
% \subsection{Mini-batching and noise models}
% Assume an unbiased stochastic gradient $g(x,\xi)$ with a second-moment bound
% \[
% \E\|g(x,\xi)-\nabla f(x)\|^2 \le \sigma^2.
% \]
% For a mini-batch of size $b$ with i.i.d.\ samples,
% $\bar g_b(x)=\frac1b\sum_{i=1}^b g(x,\xi_i)$ remains unbiased and satisfies
% \[
% \E\|\bar g_b(x)-\nabla f(x)\|^2 \le \frac{\sigma^2}{b}.
% \]

% In the SDE-based view used in Assumption~\ref{ass:noise}, stochastic perturbations are represented by a diffusion covariance \(\Sigma\), so that over a step of length \(\alpha_k\) the injected perturbation has covariance \(\alpha_k\Sigma\). In the isotropic case we write \(\Sigma=\rho^2 I\), and the variance proxy used in the bounds is then \(\sigma^2=\tr(\Sigma)=n\rho^2\).

% Under the usual i.i.d.\ mini-batch variance reduction, increasing the batch size from \(1\) to \(b\) corresponds to
% \[
% \Sigma \mapsto \Sigma/b,
% \qquad\text{and therefore}\qquad
% \sigma^2=\tr(\Sigma) \mapsto \sigma^2/b.
% \]
% All bounds depending on the variance proxy inherit the same scaling.

% In particular, the steady-state mean-square error bound from Theorem~\ref{thm:main} becomes
% \[
% \limsup_{k\to\infty}\E\|x_k-x^\star\|^2 \;\lesssim\; \frac{C}{\alpha}\cdot\frac{\sigma^2}{b},
% \]
% highlighting a $1/b$ reduction of the stationary mean-square error bound under this proxy.

\subsection{Mini-batching and stochastic-oracle variants}

The analysis above is formulated for the diffusion/center-noise model:
after the implicit discretization, the stochastic perturbation appears as
an additive perturbation of the resolvent center,
\[
x_{k+1}
=
\prox_{\lambda_k f}(c_k+\xi_k),
\qquad
\E\|\xi_k\|^2
\le
\frac{\alpha_k}{(1+\tau_k)^2}\sigma^2.
\]
Within this model, any mechanism that reduces the effective covariance
of the center perturbation by a factor \(b\) immediately reduces the
variance proxy as
\[
\sigma^2 \mapsto \frac{\sigma^2}{b}.
\]
Consequently, the stationary mean-square bound from
Theorem~\ref{thm:main} scales formally as
\[
\limsup_{k\to\infty}\E\|x_k-x^\star\|^2
\;\lesssim\;
\frac{C}{\alpha}\cdot\frac{\sigma^2}{b}.
\]

This should be distinguished from the standard mini-batch stochastic
gradient model. If \(g(x,\xi)\) is an unbiased stochastic gradient and
\[
\E\|g(x,\xi)-\nabla f(x)\|^2\le \sigma^2,
\]
then the mini-batch estimator
\[
\bar g_b(x)=\frac1b\sum_{i=1}^b g(x,\xi_i)
\]
satisfies the usual variance reduction
\[
\E\|\bar g_b(x)-\nabla f(x)\|^2\le \frac{\sigma^2}{b}.
\]
However, in that case the randomness enters through the gradient or
through a stochastic mini-batch objective used inside the implicit
resolvent solve. This gives a related but distinct stochastic-oracle
method, not exactly the center-noise recursion analyzed in
Theorem~\ref{thm:main}.

Thus, the \(1/b\) scaling above should be read as a center-noise proxy:
if the effective perturbation of the resolvent center has covariance
reduced by a factor \(b\), then the same reduction appears in the
stationary bound. A complete theorem for fully stochastic mini-batch
inner solves would require additional assumptions on the stochastic
oracle, its dependence on the inner iterates, and the accuracy of the
inexact solve. We leave this extension to future work.

% ===================== Algorithms =====================
\section{Algorithms}\label{sec:algorithms}
% Algorithms~\ref{alg:ironfi} and~\ref{alg:inner-lm} provide, respectively, pseudocode for the \IRONfi{} outer step and for the LM/Newton inner solve used to compute the resolvent point
% \(
% x_{k+1}=\prox_{\lambda_k f}(c_k+\xi_k).
% \)

Algorithms~\ref{alg:ironfi} and~\ref{alg:inner-lm} provide,
respectively, pseudocode for the \IRONfi{} outer step under the diffusion/center-noise model and for the LM/Newton inner solve used to
compute the resolvent point
\[
x_{k+1}
=
\prox_{\lambda_k f}(c_k+\xi_k).
\]
Algorithm~\ref{alg:ironfi} should therefore be read as the algorithmic
form of the SDE discretization analyzed in the theory. In controlled
simulations, as in Step 5 in Algorithm~\ref{alg:ironfi}, \(\xi_k\) can be generated from the Gaussian diffusion model.
In stochastic-oracle implementations, such as mini-batch learning
variants, the randomness may instead enter through the gradients or
curvature information used inside the inner solve; this is a practical
variant of the resolvent viewpoint, but not the exact recursion covered
by the center-noise theorem.

\begin{algorithm}[ht!]
%\caption{\IRONfi: Implicit Resolvent Optimization under Noise (outer loop)}
\caption{\IRONfi: implicit resolvent step under the diffusion/center-noise model}
\label{alg:ironfi}
\begin{algorithmic}[1]
\STATE \textbf{Input:} $x_0\in\R^n$, $v_0\in\R^n$, $\gamma_0>0$, stepsizes $\{\alpha_k\}_{k\ge0}$, strong convexity parameter $\mu$, covariance factor $\Sigma^{1/2}$, with variance proxy $\sigma^2=\tr(\Sigma)$; in the isotropic case, $\Sigma=\rho^2 I$.
\FOR{$k=0,1,2,\dots$}
  \STATE Define $\displaystyle \tau_k=\frac{1}{\alpha_k}+\frac{\mu}{\gamma_k}$,\quad $\displaystyle \lambda_k=\frac{\alpha_k}{\gamma_k(1+\tau_k)}$.
  %\STATE Center $c_k=\frac{v_k+\tau_k x_k}{1+\tau_k}$.
  \STATE {\color{black} Compute the }center
\[
c_k=\frac{v_k+\tau_k x_k}{1+\tau_k}
=
x_k+\frac{v_k-x_k}{1+\tau_k}.
\]
  % \STATE Sample Brownian increment: $W_{k+1}-W_k=\sqrt{\alpha_k}\,\eta_k$, with $\eta_k\sim\mathcal N(0,I)$.
  % \STATE Center perturbation: $\displaystyle \xi_k=\frac{\Sigma^{1/2}\sqrt{\alpha_k}}{1+\tau_k}\,\eta_k$.
  \STATE Obtain a center perturbation \(\xi_k\) consistent with the diffusion model. For the Gaussian model analyzed in the theory, draw \(\eta_k\sim\mathcal N(0,I)\) and set
\[
\xi_k
=
\frac{\Sigma^{1/2}\sqrt{\alpha_k}}{1+\tau_k}\,\eta_k .
\]

  \STATE Compute $x_{k+1}$ by approximately solving the resolvent (proximal) subproblem
  \[
     x_{k+1} \approx \prox_{\lambda_k f}(c_k+\xi_k)
     \quad\text{via Alg.~\ref{alg:inner-lm} (LM/Newton).}
  \]
  \STATE Update velocity (implicit Euler coupling):
  \[
     v_{k+1}=x_{k+1}+\frac{x_{k+1}-x_k}{\alpha_k}.
  \]
  \STATE Update damping: $\displaystyle \gamma_{k+1}=\frac{\gamma_k+\alpha_k\mu}{1+\alpha_k}$.
\ENDFOR
\end{algorithmic}
\end{algorithm}

\begin{algorithm}[ht!]
\caption{LM/Newton inner solve for $x=\prox_{\lambda f}(c)$ (one outer iteration)}
\label{alg:inner-lm}
\begin{algorithmic}[1]
\STATE \textbf{Input:} center $c=c_k+\xi_k$, parameter $\lambda=\lambda_k$, function $f$ with gradient $\nabla f$ and Hessian (or GN approximation) $H(x)$; residual tolerance $\delta_k>0$; max iters $N_{\max}$.
% \STATE Define fixed-point residual $g(u)=u-c+\lambda\nabla f(u)$; Jacobian $J(u)=I+\lambda H(u)$.
%\STATE Define the fixed-point residual
%\[
%g(u)=u-c+\lambda\nabla f(u).
%\]

\STATE Initialize \(u^{(0)}\leftarrow x_k\) if a warm start is available; otherwise set \(u^{(0)}\leftarrow c\).
\FOR{$i=0,1,\dots,N_{\max}$}
  \STATE Compute the fixed-point residual $g_i=g(u^{(i)}) = u^{(i)}-c+\lambda\nabla f(u^{(i)})$; if $\|g_i\|\le \delta_k$ \textbf{return} $x_{k+1}=u^{(i)}$.
  % \STATE (LM/Newton step) Solve $(I+\lambda H(u^{(i)}))\,s^{(i)}=-g_i$ \textit{(direct, or matrix-free CG/MINRES using Hessian--vector products)}.
  \STATE {\color{black} Form the Jacobian/LM matrix}
\[
J_i=I+\lambda H(u^{(i)}).
\]
  \STATE (LM/Newton step) Solve
\[
J_i s^{(i)}=-g_i,
\]
\textit{directly, or matrix-free by CG/MINRES using Hessian--vector products}.

  \STATE (Optional damping / trust region) If $\|g(u^{(i)}+s^{(i)})\|>\|g_i\|$, shrink $s^{(i)}\leftarrow\beta s^{(i)}$ with $\beta\in(0,1)$ until decrease holds.
  \STATE Update $u^{(i+1)}\leftarrow u^{(i)}+s^{(i)}$.
\ENDFOR
\STATE \textbf{Return} $x_{k+1}=u^{(N_{\max})}$.
\end{algorithmic}
\end{algorithm}

\begin{remark}[Algorithmic considerations: inner solves, scalability, and practical variants.]

The fully implicit outer step of \IRONfi{} reduces each iteration to computing a resolvent point
\(
x_{k+1}=\prox_{\lambda_k f}(c_k+\xi_k),
\)
which is the unique root of the fixed-point residual
\[
g_k(u)\;:=\;u-(c_k+\xi_k)+\lambda_k\nabla f(u),
\quad\text{so that}\quad g_k(x_{k+1})=0.
\]
In practice, we compute an approximation $\tilde x_{k+1}$ by running a few LM/Newton iterations, each requiring the solution of a linear system
\[
\big(I+\lambda_k H(u)\big)s = -g_k(u),
\]
where $H(u)$ is either the true Hessian $\nabla^2 f(u)$ (when available) or a Gauss--Newton / Fisher-type curvature approximation. To scale to large problems, these inner systems can be solved \emph{matrix-free} using Krylov methods (CG for SPD models, MINRES/GMRES otherwise), driven by Hessian--vector products; this makes the per-outer-iteration cost essentially proportional to the number of inner Krylov iterations times the cost of an $H(u)$--vector product. Warm-starts are natural and effective: we initialize the inner loop at $u^{(0)}=x_k$ (or reuse the last inner iterate), which substantially reduces inner work when the outer iterates move smoothly.

A key advantage of the resolvent formulation is that it provides an implementable and theory-aligned stopping rule. Since $f$ is $\mu$-strongly convex, $g_k$ is $(1+\lambda_k\mu)$-strongly monotone, hence
\[
\|\tilde x_{k+1}-x_{k+1}\|
\;\le\;
\frac{\|g_k(\tilde x_{k+1})\|}{1+\lambda_k\mu}.
\]
Therefore, controlling the inner residual directly controls the distance to the exact resolvent point. Combined with the scaling $(1+\lambda_k\mu)^{-2}=O(\alpha_k^{-2})$ (uniformly over $k$ under bounded $\gamma_k$), this implies that even a moderate residual tolerance can yield an inner error that is negligible compared to the \(O(1/\alpha_k)\) stochastic contribution in the stationary mean-square bound. Concretely, if the inner solver returns $\tilde x_{k+1}$ such that $\E\|g_k(\tilde x_{k+1})\|^2\le \delta_k^2$ with $\delta_k=O(1)$, then
\(
\E\|\tilde x_{k+1}-x_{k+1}\|^2 = O(\alpha_k^{-2}),
\)
which is dominated by the \(O(1/\alpha_k)\) stochastic contribution controlled by the theory. This observation supports a practical regime in which inner solves are intentionally \emph{inexact} yet do not compromise the \(O(1/\alpha_k)\) stationary mean-square scaling.

Finally, several low-cost variants fit naturally within the same framework. One may replace $H(u)$ by a diagonal or block-diagonal surrogate (yielding a cheap LM step), or use a quasi-Newton approximation (e.g., limited-memory updates) inside the same residual-based stopping rule. Another pragmatic option is to perform only one (or a small, fixed number of) LM/Newton steps per outer iteration, which can be viewed as an inexact resolvent evaluation; the above residual-to-error bound then provides a direct diagnostic for when such truncated inner solves remain sufficient.

\end{remark}

\color{black}
\section{Numerical experiments}\label{sec:experiments}

We now test the main mechanism predicted by the theory: for the fully implicit resolvent scheme \IRONfi{}, increasing the implicit stepsize strengthens the resolvent contraction and reduces the stationary spread induced by noise. The experiments focus on the phenomena predicted by the analysis: stationary mean-square scaling, the role of the inexact inner solve, and the practical cost of computing the resolvent step.

\paragraph{Toy particle simulations (Sections~\ref{subsec:SCVX}--\ref{subsec:logcosh}).}
We first study two low-dimensional problems in \(\R^3\): a strongly convex quadratic with known minimizer \(x^\star\), and a nonconvex log-cosh regression example.
For the quadratic problem, the theoretical quantities can be measured directly. We therefore report the ensemble mean-square error
\[
\widehat{\mathrm{MSE}}_k
:=
\frac1N\sum_{j=1}^N\|x_k^{(j)}-x^\star\|^2
\]
and its bias--variance decomposition
\[
\widehat{\mathrm{MSE}}_k
=
\|\bar x_k-x^\star\|^2
+
\tr\!\bigl(\widehat{\mathrm{Cov}}(x_k)\bigr).
\]
Here \(\widehat{\mathrm{Cov}}(x_k)\) denotes the empirical covariance with normalization \(1/N\),
\[
\widehat{\mathrm{Cov}}(x_k)
:=
\frac1N\sum_{j=1}^N
\bigl(x_k^{(j)}-\bar x_k\bigr)
\bigl(x_k^{(j)}-\bar x_k\bigr)^\top,
\]
so that the displayed bias--variance decomposition is exact.

We also compare the empirical stationary value of \(\alpha\,\widehat{\mathrm{MSE}}_\infty\) with the exact stationary covariance predicted by the discrete Lyapunov equation in the fixed-\(\gamma\) quadratic recursion.
For the nonconvex log-cosh example, the goal is only qualitative: we visualize late-time particle clouds to illustrate how the stationary spread changes with \(\alpha\) around the critical regions reached by the dynamics.

\paragraph{Learning benchmarks (Sections~\ref{sec:exp-logreg}--\ref{sec:exp-mnist}).}
We then move to learning problems.
First, we study synthetic \(\ell_2\)-regularized logistic regression, a strongly convex benchmark where a high-accuracy reference minimizer is available and stationary quantities can be estimated over multiple random seeds.
This experiment tests the \(O(1/\alpha)\) stationary-MSE scaling, the robustness of this scaling to fixed inner residual tolerances, and the practical inner-iteration cost.
Second, we consider MNIST softmax regression and compare \IRONfi{} with AdamW and NAG-GS under a validation-based tuning protocol.
In that benchmark, the emphasis is on accuracy, stability, and wall-clock cost, including the overhead of the inexact Newton--CG inner solve.
\color{black}

\color{black}
\subsection{Strongly convex quadratic}\label{subsec:SCVX}

We take
\[
f(x)=\frac12 x^\top A x-b^\top x,
\qquad
A=Q^\top \operatorname{diag}(1,1,3)Q,
\qquad
b=c\,\mathbf 1,
\]
where \(Q\) is orthogonal. The minimizer is \(x^\star=A^{-1}b\).
For this experiment we use the fixed-\(\gamma\) setting, so that the quadratic recursion is linear after shifting by \(x^\star\). This allows us to compare Monte Carlo particle simulations with the exact stationary covariance obtained from the discrete Lyapunov equation in Section~\ref{sec:quad-lyapunov}.

Each \IRONfi{} step is explicit in closed form:
\[
x_{k+1}
=
(I+\lambda A)^{-1}(c_k+\xi_k),
\qquad
\lambda=\frac{\alpha}{\gamma(1+\tau)},
\]
with
\[
c_k=\frac{v_k+\tau x_k}{1+\tau},
\qquad
\tau=\frac1\alpha+\frac{\mu}{\gamma}.
\]
We run \(N=2\times 10^5\) independent particles and vary the implicit stepsize parameter over the four values shown in the figures.

\paragraph{Mean-square error and bias--variance decomposition.}
At each iteration we compute
\[
\widehat{\mathrm{MSE}}_k
=
\frac1N\sum_{j=1}^N
\|x_k^{(j)}-x^\star\|^2.
\]
We also decompose this quantity as
\[
\widehat{\mathrm{MSE}}_k
=
\|\bar x_k-x^\star\|^2
+
\tr\!\bigl(\widehat{\mathrm{Cov}}(x_k)\bigr),
\qquad
\bar x_k:=\frac1N\sum_{j=1}^N x_k^{(j)}.
\]
This is the quantity appearing in the mean-square theory: the ensemble mean alone may be small even if the stationary cloud remains spread out.

\begin{figure}[ht!]
  \centering

  \begin{subfigure}[t]{0.48\linewidth}
    \includegraphics[width=\linewidth]{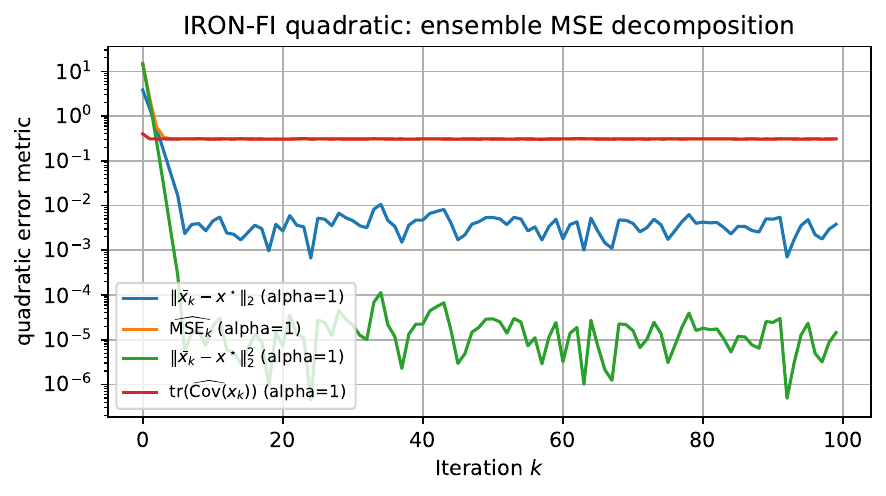}
    \caption{\(\alpha=1\)}
  \end{subfigure}\hfill
  \begin{subfigure}[t]{0.48\linewidth}
    \includegraphics[width=\linewidth]{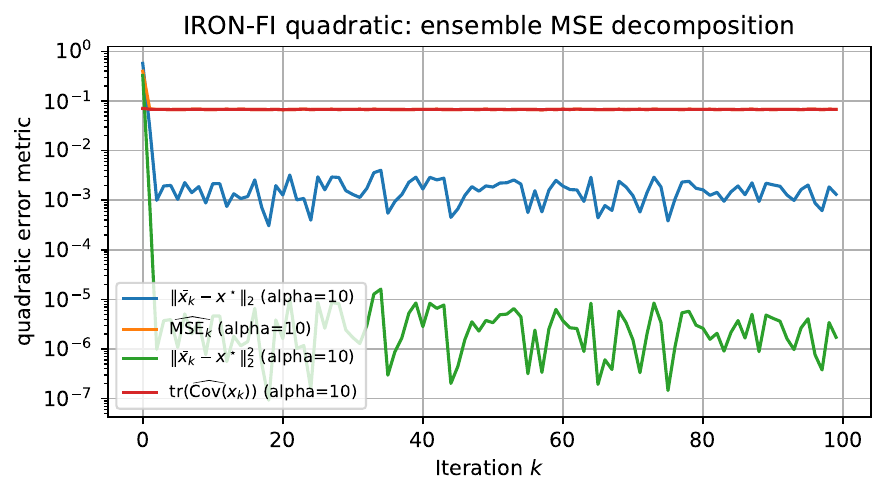}
    \caption{\(\alpha=10\)}
  \end{subfigure}

  \par\smallskip

  \begin{subfigure}[t]{0.48\linewidth}
    \includegraphics[width=\linewidth]{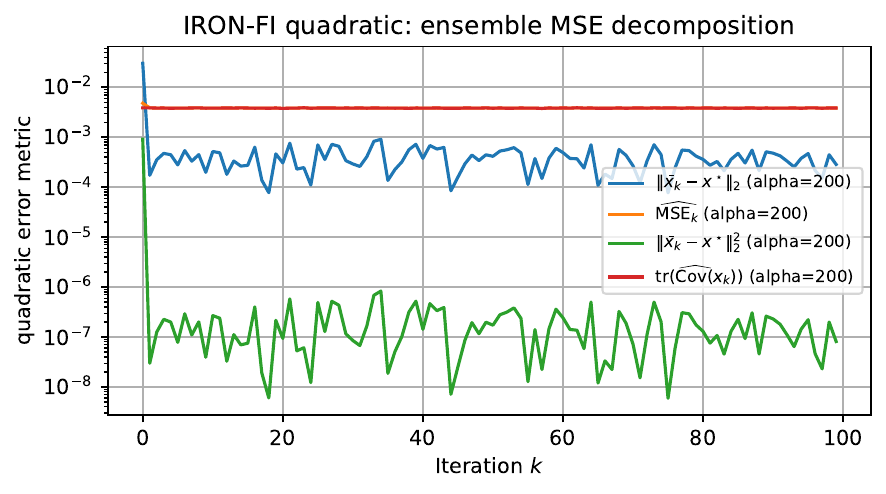}
    \caption{\(\alpha=200\)}
  \end{subfigure}\hfill
  \begin{subfigure}[t]{0.48\linewidth}
    \includegraphics[width=\linewidth]{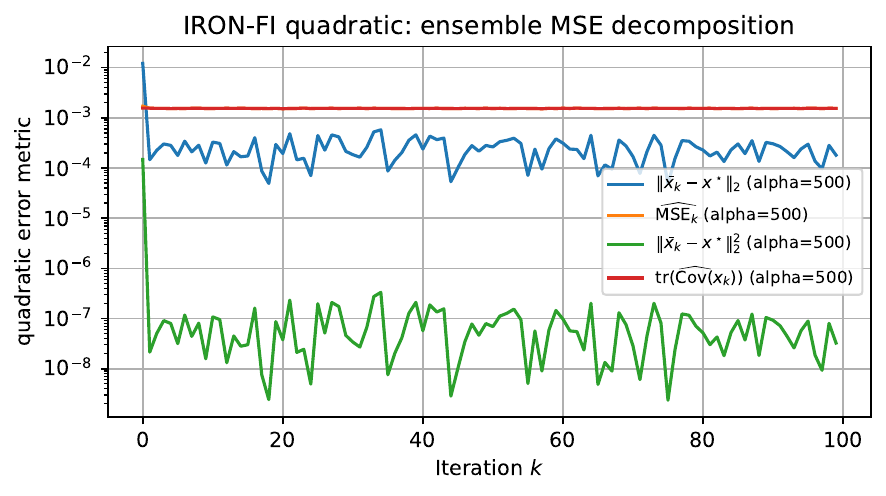}
    \caption{\(\alpha=500\)}
  \end{subfigure}

  \caption{Quadratic test: evolution of the ensemble mean-square error and its bias--variance decomposition. The plotted quantities are \(\widehat{\mathrm{MSE}}_k\), the squared bias \(\|\bar x_k-x^\star\|^2\), and the covariance trace \(\tr(\widehat{\mathrm{Cov}}(x_k))\). As the implicit stepsize increases, both the transient and the late-time mean-square level decrease.}
  \label{fig:quad-mse-decomp}
\end{figure}

Figure~\ref{fig:quad-mse-decomp} shows that increasing \(\alpha\) reduces the late-time mean-square level. The decomposition also shows that the covariance trace is the relevant stationary quantity once the mean has approached \(x^\star\). This directly matches the mean-square quantities controlled in Theorem~\ref{thm:main} and the sharper quadratic analysis of Proposition~\ref{prop:quad-explicit}.

\paragraph{Exact stationary covariance and asymptotic constant.}
For the fixed-\(\gamma\) quadratic recursion, the shifted state
\[
z_k=
\begin{pmatrix}
e_k\\ w_k
\end{pmatrix},
\qquad
e_k=x_k-x^\star,
\qquad
w_k=v_k-x^\star,
\]
satisfies a linear stochastic recursion
\[
z_{k+1}=Mz_k+G\xi_k.
\]
The stationary covariance \(P=\E[z_kz_k^\top]\) solves
\[
P=MPM^\top+Q,
\]
and the exact stationary position covariance is the top-left block \(P_{xx}\). Thus
\[
\mathrm{MSE}^{\rm exact}_\infty
=
\tr(P_{xx}).
\]
In the isotropic case \(\Sigma=\rho^2 I\), Proposition~\ref{prop:quad-explicit} predicts
\[
\alpha\,\E\|x_\infty-x^\star\|^2
\longrightarrow
C_{\rm quad}
:=
\gamma^2\rho^2\tr(A^{-2})
=
\frac{\gamma^2}{n}\sigma^2\tr(A^{-2}),
\]
where \(\sigma^2=\tr(\Sigma)=n\rho^2\).

\begin{figure}[ht!]
  \centering
  \includegraphics[width=0.78\linewidth]{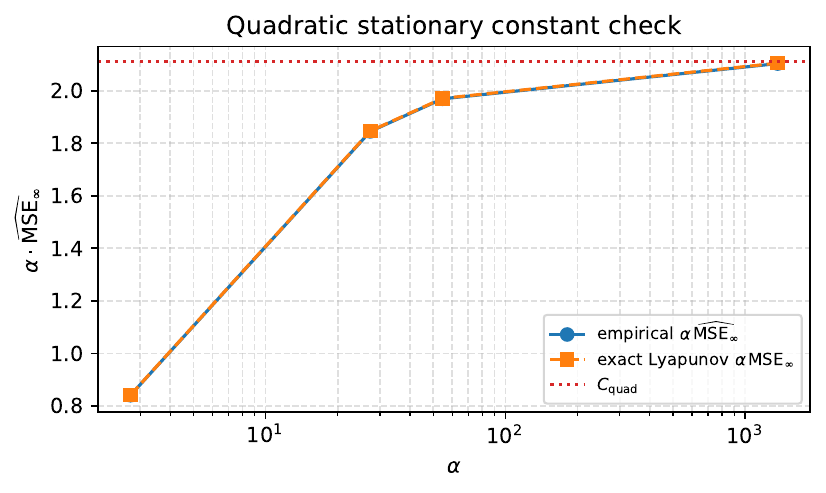}
  \caption{Quadratic test: scaled stationary mean-square error. The empirical curve is the late-time Monte Carlo estimate of \(\alpha\,\widehat{\mathrm{MSE}}_\infty\). The exact Lyapunov curve is \(\alpha\,\tr(P_{xx})\), where \(P\) solves the discrete Lyapunov equation for the fixed-\(\gamma\) linear quadratic recursion. The horizontal line is the asymptotic constant \(C_{\rm quad}\). The agreement is consistent with the predicted \(1/\alpha\) stationary scaling and the sharper quadratic constant.}
  \label{fig:quad-scaled-mse}
\end{figure}

Figure~\ref{fig:quad-scaled-mse} gives a direct theory--experiment comparison. The empirical stationary estimate is obtained by averaging \(\widehat{\mathrm{MSE}}_k\) over a late-time window after discarding the initial transient part of the trajectory and then plotting \(\alpha\,\widehat{\mathrm{MSE}}_\infty\). The exact curve is computed independently from the discrete Lyapunov equation of the linear recursion. The convergence of both curves toward \(C_{\rm quad}\) illustrates the asymptotic statement of Proposition~\ref{prop:quad-explicit}.

\paragraph{Stationary clouds.}
Finally, we visualize the particle clouds in the three coordinate planes. These plots are qualitative, but they make the reduction of the stationary spread around \(x^\star\) clearly visible.

\begin{figure}[ht!]
  \centering

  \begin{subfigure}[t]{\linewidth}
    \includegraphics[width=\linewidth]{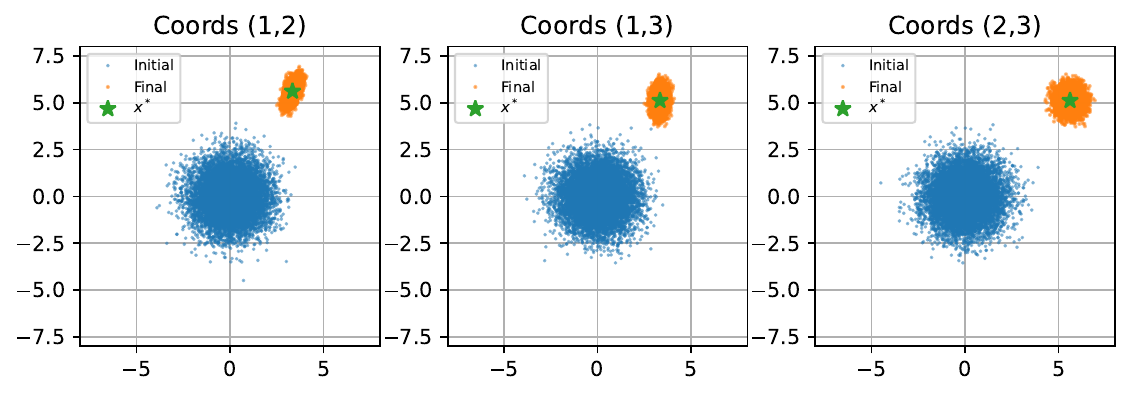}
    \caption{\(\alpha=1\)}
  \end{subfigure}

  \begin{subfigure}[t]{\linewidth}
    \includegraphics[width=\linewidth]{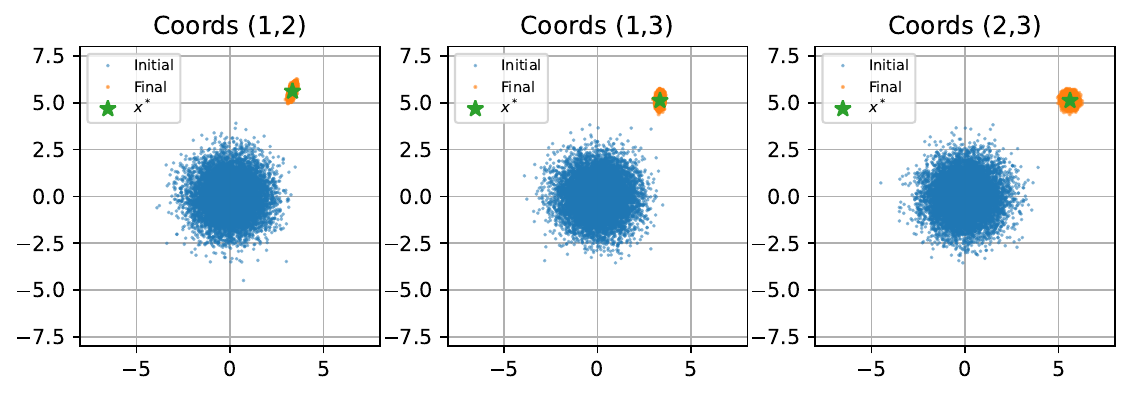}
    \caption{\(\alpha=10\)}
  \end{subfigure}

  \begin{subfigure}[t]{\linewidth}
    \includegraphics[width=\linewidth]{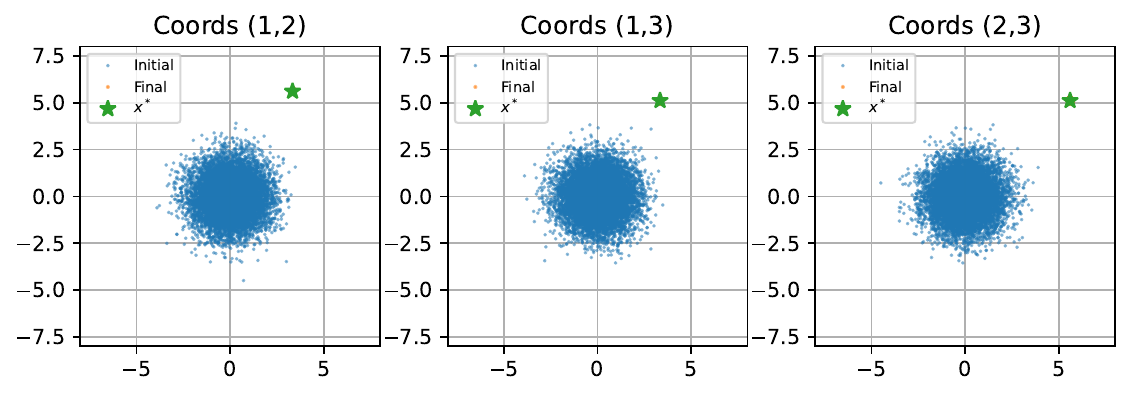}
    \caption{\(\alpha=200\)}
  \end{subfigure}

  \begin{subfigure}[t]{\linewidth}
    \includegraphics[width=\linewidth]{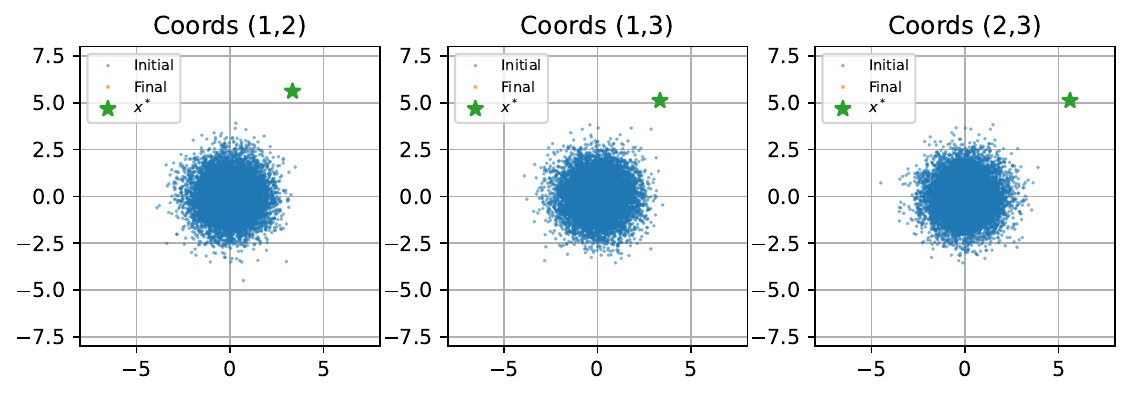}
    \caption{\(\alpha=500\)}
  \end{subfigure}

  \caption{Quadratic test: projected point clouds (blue = initial, orange = final, star = minimizer) for four implicit stepsizes. Each subfigure shows the coordinate planes \((x_1,x_2)\), \((x_1,x_3)\), and \((x_2,x_3)\). As \(\alpha\) grows, the stationary cloud concentrates around \(x^\star\), visually reflecting the \(O(1/\alpha)\) stationary mean-square error bound.}
  \label{fig:quad-clouds-2x2}
\end{figure}

\noindent\textbf{Observation.}
% The quadratic experiment now verifies both levels of the theory. First, the ensemble mean-square error decreases with \(\alpha\), and its late-time value is dominated by the covariance trace once the bias has decayed. Second, the scaled stationary quantity \(\alpha\,\widehat{\mathrm{MSE}}_\infty\) matches the exact fixed-\(\gamma\) Lyapunov prediction and approaches the explicit constant \(C_{\rm quad}\). The point-cloud plots provide a complementary visual confirmation: for large \(\alpha\), the stationary spread around \(x^\star\) becomes barely visible.
The quadratic experiment illustrates both levels of the theory. First, the ensemble mean-square error decreases with \(\alpha\), and its late-time value is dominated by the covariance trace once the bias has decayed. Second, the scaled stationary quantity \(\alpha\,\widehat{\mathrm{MSE}}_\infty\) matches the exact fixed-\(\gamma\) Lyapunov prediction and approaches the explicit constant \(C_{\rm quad}\). The point-cloud plots provide a complementary visual illustration: for large \(\alpha\), the stationary spread around \(x^\star\) becomes barely visible.

\subsection{Nonconvex log-cosh regression}\label{subsec:logcosh}

We next consider a mildly nonconvex example to illustrate whether the same variance-shrinking behavior can be observed locally beyond the strongly convex setting. This experiment is qualitative: the theory developed in Section~\ref{sec:constant} assumes global strong convexity, so the plots below should be interpreted as visual evidence of a local phenomenon rather than as a theorem-level validation.

We consider
\[
f(x)
=
\frac12\|A u(x)-b\|^2,
\qquad
u_i(x)=\log\cosh(x_i),
\]
with gradient
\[
\nabla f(x)
=
\bigl(Q u(x)-c\bigr)\odot \tanh(x),
\qquad
Q=A^\top A,
\qquad
c=A^\top b.
\]
The fully implicit update is
\[
x_{k+1}=c_k-\lambda\nabla f(x_{k+1}),
\]
and the inner nonlinear system is solved by LM/Newton iterations with Jacobian
\[
I+\lambda \nabla^2 f(x_{k+1}).
\]
We use the same values of \(\alpha\) as in the quadratic experiment and initialize the particles in a region from which the late-time dynamics remain near the relevant stationary set.

\begin{figure}[ht!]
  \centering

  \begin{subfigure}[t]{\linewidth}
    \includegraphics[width=\linewidth]{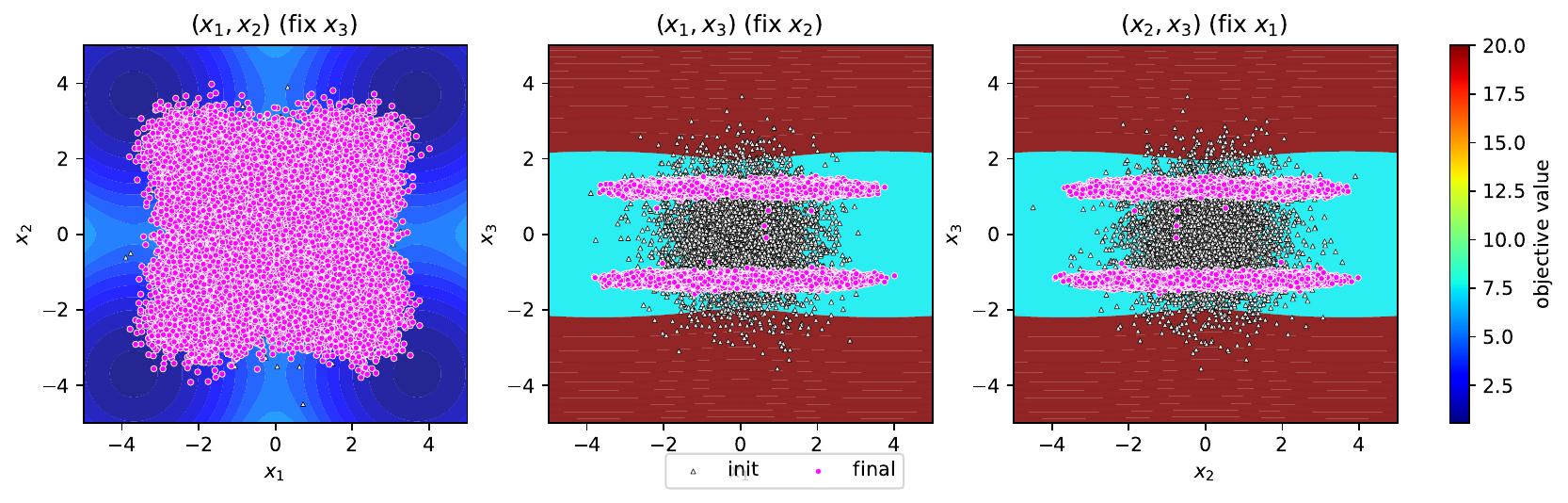}
    \caption{\(\alpha=1\)}
  \end{subfigure}

  \begin{subfigure}[t]{\linewidth}
    \includegraphics[width=\linewidth]{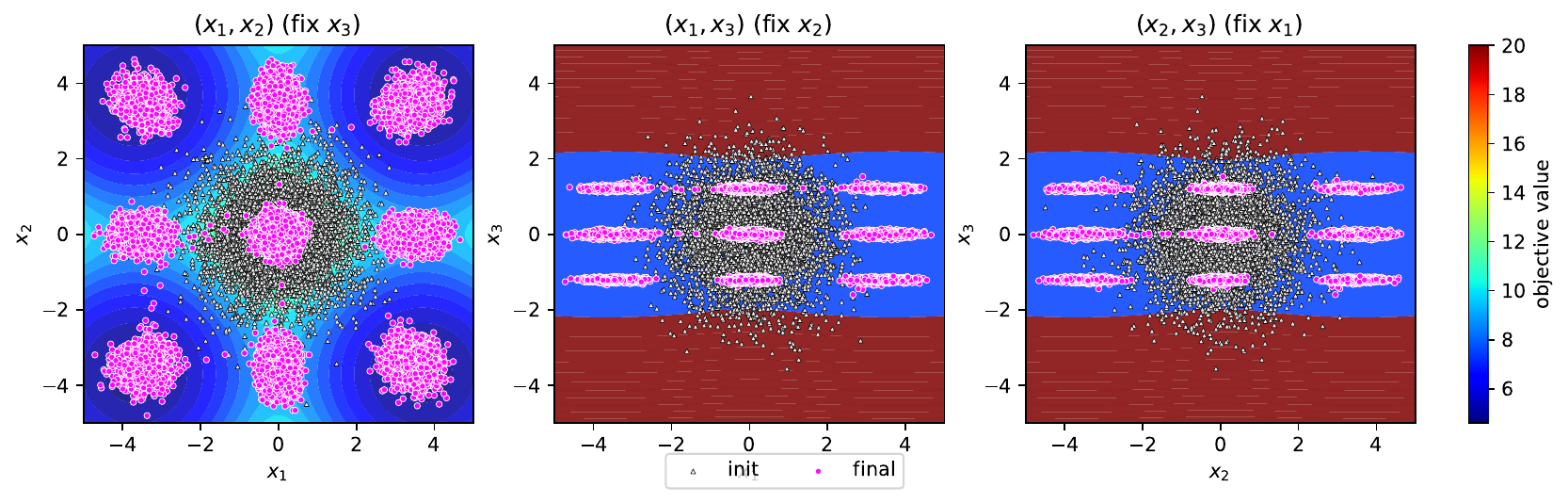}
    \caption{\(\alpha=10\)}
  \end{subfigure}

  \begin{subfigure}[t]{\linewidth}
    \includegraphics[width=\linewidth]{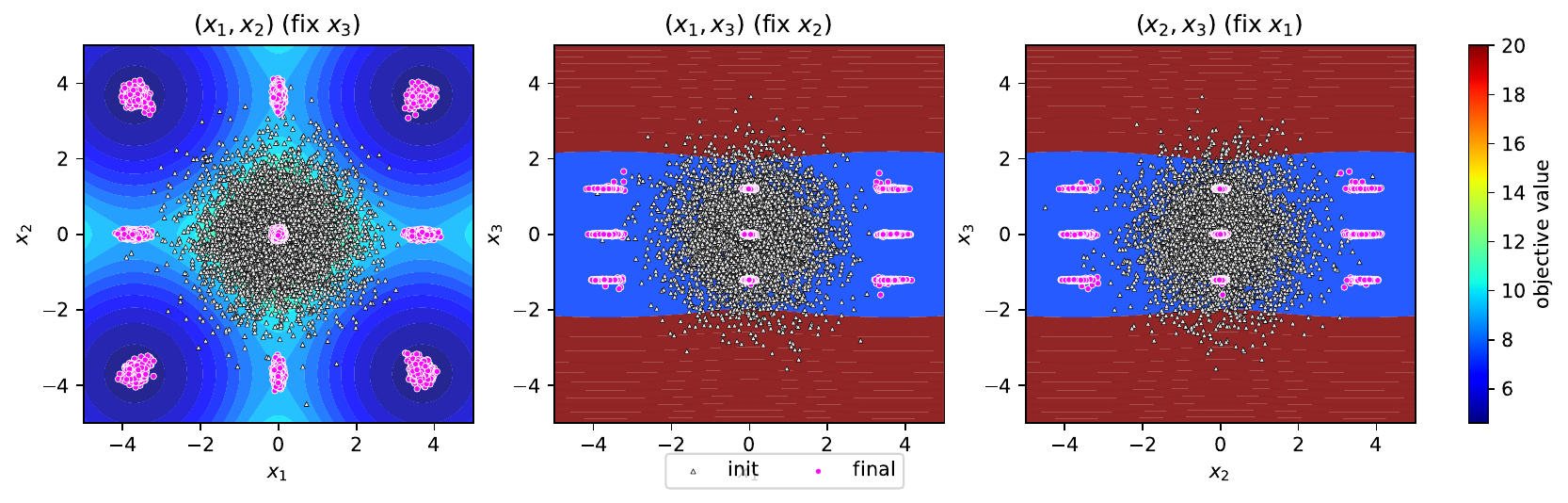}
    \caption{\(\alpha=200\)}
  \end{subfigure}

  \begin{subfigure}[t]{\linewidth}
    \includegraphics[width=\linewidth]{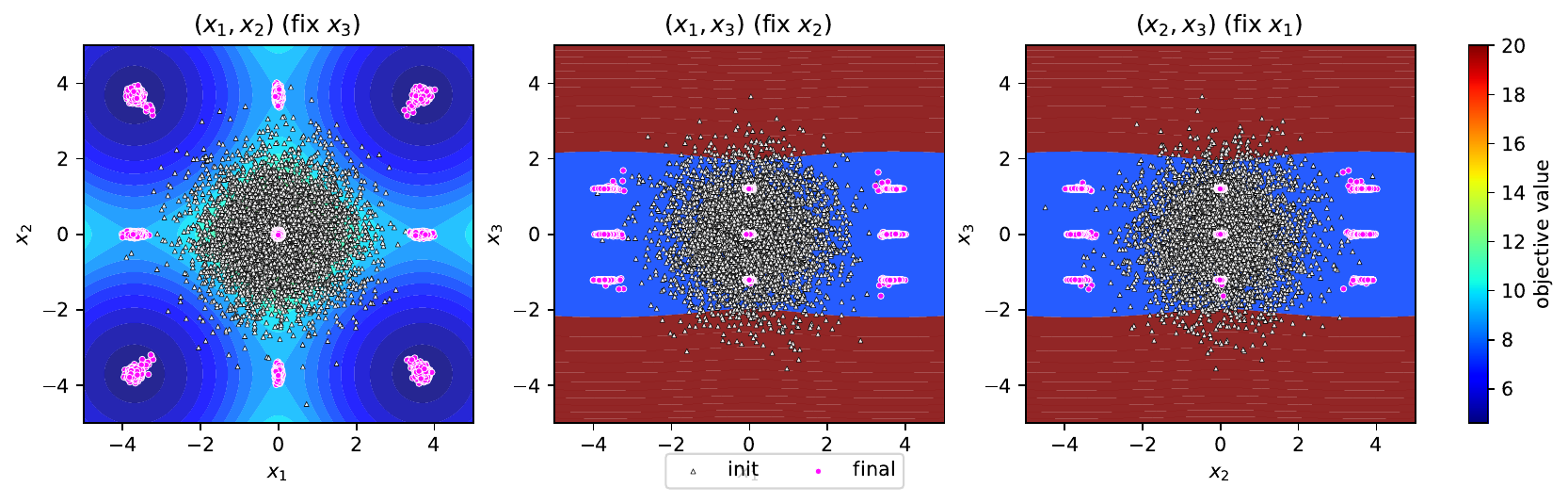}
    \caption{\(\alpha=500\)}
  \end{subfigure}

  \caption{Nonconvex log-cosh regression: projected particle clouds for four implicit stepsizes. Each subfigure shows the coordinate planes \((x_1,x_2)\), \((x_1,x_3)\), and \((x_2,x_3)\). The plotted clouds are obtained from late-time iterates, after discarding the initial transient phase. As \(\alpha\) increases, the clouds become visibly tighter around the stationary regions reached by the dynamics. This is a qualitative, beyond-theory illustration rather than a consequence of Theorem~\ref{thm:main}.}
  \label{fig:logcosh-clouds}
\end{figure}

\noindent\textbf{Observation.}
In this nonconvex example, the particles concentrate near critical regions of the log-cosh objective, including neighborhoods of global minimizers and, for some initializations, non-optimal critical regions. Increasing \(\alpha\) visibly reduces the late-time spread of the clouds. This suggests that the variance-shrinking effect induced by the implicit resolvent may persist locally beyond the globally strongly convex regime, at least around the regions reached by the dynamics. A corresponding nonconvex theory is left for future work.
% \color{black}

\subsection{Strongly convex benchmark: synthetic regularized logistic regression}
\label{sec:exp-logreg}

We next test the stationary mean-square predictions in a controlled strongly convex learning problem. The goal of this experiment is not to compare against other optimizers, but to isolate the mechanism predicted by the theory: the resolvent contraction should make the stationary mean-square error decrease like \(1/\alpha\), and this behavior should persist when the inner nonlinear solves are performed only approximately.

\paragraph{Problem.}
We consider ridge-regularized binary logistic regression,
\[
f(w)
:=
\frac1n\sum_{i=1}^n
\log\!\left(1+\exp(-y_i a_i^\top w)\right)
+
\frac{\lambda_{\rm reg}}{2}\|w\|^2,
\]
where \(a_i\in\R^d\), \(y_i\in\{-1,+1\}\), and \(\lambda_{\rm reg}>0\). The regularization makes \(f\) \(\mu\)-strongly convex with \(\mu=\lambda_{\rm reg}\). We generate synthetic Gaussian features and logistic labels, and compute a high-accuracy reference minimizer \(w^\star\) once using a deterministic full-batch solver.

\paragraph{Noise model and resolvent step.}
To match the mechanism derived in Section~\ref{sec:scheme}, we do not use mini-batch gradients in this synthetic experiment. Instead, stochasticity is injected directly as a Gaussian perturbation of the resolvent center. At iteration \(k\), we form
\[
c_k
=
\frac{v_k+\tau x_k}{1+\tau},
\qquad
\tau=\frac1\alpha+\frac{\mu}{\gamma},
\]
sample \(\eta_k\sim\mathcal N(0,I_d)\), and set, in the isotropic case,
\[
\xi_k
=
\frac{\sqrt{\alpha}}{1+\tau}\,\rho\,\eta_k.
\]
Thus
\[
\E[\xi_k]=0,
\qquad
\E\|\xi_k\|^2
=
\frac{\alpha}{(1+\tau)^2}\rho^2 d.
\]
The next iterate is computed by approximately solving
\[
w_{k+1}
\approx
\prox_{\lambda^{\rm prox} f}(c_k+\xi_k),
\qquad
\lambda^{\rm prox}
=
\frac{\alpha}{\gamma(1+\tau)}.
\]
The inner solve is performed by LM/Newton iterations applied to the fixed-point residual
\[
g_k(u)
:=
u-(c_k+\xi_k)+\lambda^{\rm prox}\nabla f(u).
\]
The inner loop stops when
\[
\|g_k(u)\|\le \varepsilon,
\]
or when the maximum number of inner iterations is reached.

\paragraph{Stationary metrics and reproducibility.}
After discarding the initial transient part of the trajectory, we estimate the stationary mean-square error by averaging over a late-time window:
\[
\widehat{\mathrm{MSE}}_\infty(\alpha)
:=
\frac1T\sum_{k=K_0}^{K_0+T-1}
\|w_k-w^\star\|^2.
\]
We repeat the experiment over the seeds \(0,1,2,3,4\). For each pair \((\alpha,\varepsilon)\), we report the mean stationary MSE across seeds, together with uncertainty bands. The script also records the full run configuration: dataset size \(n\), dimension \(d\), regularization parameter \(\lambda_{\rm reg}\), noise level \(\rho\), total number of iterations, stationary averaging window, tolerance grid, stepsize grid, maximum number of inner iterations, and the random seeds.

\paragraph{Slope estimation.}
To test the predicted law
\[
\mathrm{MSE}_\infty(\alpha)=O(1/\alpha),
\]
we fit the slope of
\[
\log \widehat{\mathrm{MSE}}_\infty(\alpha)
\quad\text{versus}\quad
\log \alpha
\]
over the large-\(\alpha\) range. The fit is performed separately for each seed and then aggregated, giving a mean slope and a \(95\%\) confidence interval. A slope close to \(-1\) is the expected signature of the \(1/\alpha\) stationary regime.

\begin{figure}[ht!]
  \centering

  \begin{subfigure}[t]{0.48\linewidth}
    \includegraphics[width=\linewidth]{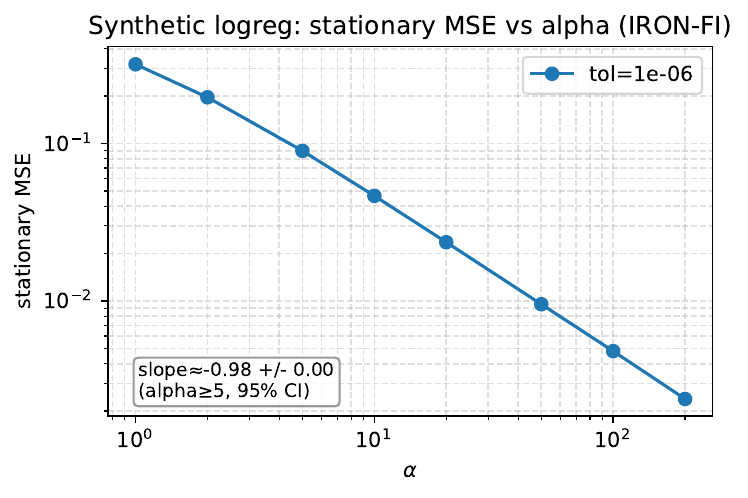}
    \caption{Stationary MSE versus \(\alpha\).}
    \label{fig:synth-logreg-mse-alpha}
  \end{subfigure}\hfill
  \begin{subfigure}[t]{0.48\linewidth}
    \includegraphics[width=\linewidth]{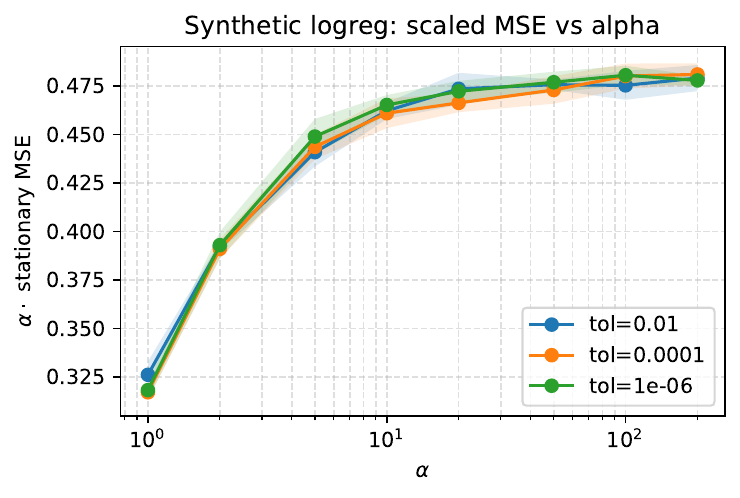}
    \caption{Scaled diagnostic \(\alpha\,\widehat{\mathrm{MSE}}_\infty\).}
    \label{fig:synth-logreg-scaled}
  \end{subfigure}

  \par\smallskip

  \begin{subfigure}[t]{0.48\linewidth}
    \includegraphics[width=\linewidth]{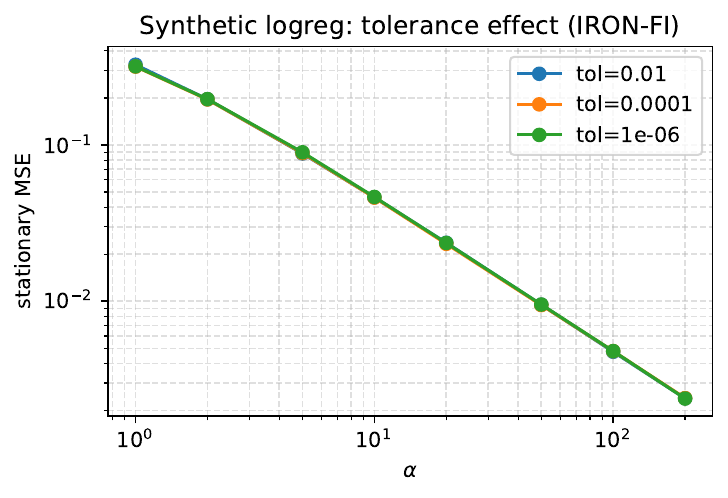}
    \caption{Effect of the inner tolerance.}
    \label{fig:synth-logreg-tol}
  \end{subfigure}\hfill
  \begin{subfigure}[t]{0.48\linewidth}
    \includegraphics[width=\linewidth]{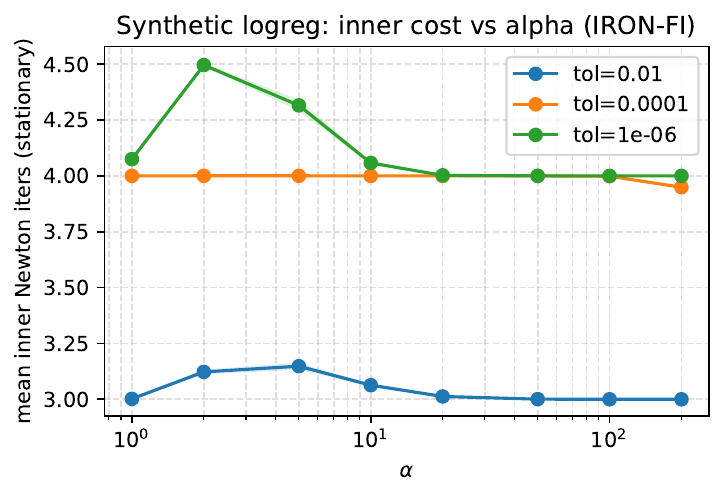}
    \caption{Mean inner LM/Newton iterations.}
    \label{fig:synth-logreg-inner}
  \end{subfigure}

  \caption{Synthetic ridge-logistic regression. Panel~(a) shows the stationary mean-square error as a function of \(\alpha\), with the fitted log--log slope reported in the plot. 
  %Panel~(b) plots \(\alpha\,\widehat{\mathrm{MSE}}_\infty\): a nearly flat curve indicates the \(1/\alpha\) regime, while deviations reveal the onset of tolerance-induced error. 
  Panel~(b) plots the scaled quantity \(\alpha\,\widehat{\mathrm{MSE}}_\infty\). 
In the ideal \(1/\alpha\) stationary regime, this quantity should approach a constant. 
Here it increases for small and moderate \(\alpha\), then levels off for larger \(\alpha\), revealing an empirical asymptotic coefficient analogous to the explicit quadratic constant \(C_{\rm quad}\), but without a closed-form expression in the present nonlinear setting.
  Panel~(c) compares several fixed residual tolerances \(\varepsilon\). Panel~(d) reports the mean number of inner LM/Newton iterations at stationarity. Curves are aggregated over five random seeds.}
  \label{fig:synth-logreg}
\end{figure}

\paragraph{Results.}
Figure~\ref{fig:synth-logreg} is consistent with the main theory-aligned behavior. First, the stationary MSE decreases approximately as \(1/\alpha\) over the large-\(\alpha\) range, with a fitted slope close to \(-1\). Second, the scaled diagnostic \(\alpha\,\widehat{\mathrm{MSE}}_\infty\) increases for small and moderate \(\alpha\), and then approaches a plateau for larger \(\alpha\). This is the nonlinear strongly convex analogue of the quadratic prediction \(\alpha\,\mathrm{MSE}_\infty\to C_{\rm quad}\): although no closed-form constant is available here, the plateau indicates an empirical coefficient \(C_{\rm logreg}\) in the large-\(\alpha\) regime. The near-overlap of the curves for different fixed inner tolerances suggests that this observed coefficient is not an artifact of the inexact inner solve, at least for the tested tolerances.
Third, the tolerance sweep shows that the inner residual tolerance does not need to shrink with \(\alpha\) to preserve the trend, provided it is not too loose. When the tolerance is too loose, the scaled-MSE plot makes the breakdown visible.

Finally, the inner-iteration plot shows that the cost of the LM/Newton solve does not grow with \(\alpha\) in this experiment; it is roughly stable for small and moderate \(\alpha\) and may even decrease in the large-\(\alpha\) regime. This supports the inexact-solve discussion of Section~\ref{sec:inexact}: the residual-to-error scaling of the resolvent allows moderate inner tolerances to remain compatible with the \(O(1/\alpha)\) stationary mean-square scaling.

\subsection{Real-data benchmark: MNIST softmax regression}
\label{sec:exp-mnist}

% We finally evaluate \IRONfi{} on a standard classification benchmark. The purpose of this experiment is practical rather than theorem-level: MNIST softmax regression is used to assess accuracy, stability, wall-clock cost, and the overhead introduced by the inexact inner solves.
We finally evaluate \IRONfi{} on a standard classification benchmark.
This experiment is practical rather than theorem-level. Unlike the controlled synthetic experiments above, the MNIST implementation is a stochastic-oracle learning variant, where mini-batches are used inside the optimization procedure. Therefore, this experiment is not intended
as a direct validation of the center-noise theorem. Its purpose is to assess whether the implicit-resolvent viewpoint leads to stable and competitive behavior in a standard learning task, and to quantify the wall-clock overhead introduced by the inexact inner solves.

\paragraph{Problem.}
We train a multiclass softmax-regression model on MNIST with \(\ell_2\) regularization. Images are flattened to \(784\)-dimensional vectors and normalized to \([0,1]\). All methods use the same preprocessing, batch sizes, regularization strength, random seeds, and epoch budget.

\paragraph{Methods.}
We compare three methods:
\begin{itemize}
\item AdamW~\cite{kingma2015adam,loshchilov2019decoupled}, as a widely used adaptive first-order baseline;
\item NAG-GS~\cite{leplat2023naggssemiimplicitacceleratedrobust}, as a semi-implicit accelerated baseline;
\item \IRONfi{}, computed with an inexact Newton--CG / LM-type inner solve and residual-based stopping.
\end{itemize}
%The comparison is not intended to establish state-of-the-art performance on MNIST. Rather, it tests whether the stability benefits suggested by the resolvent viewpoint can be obtained in a learning benchmark, and at what computational cost.
This benchmark illustrates how the resolvent viewpoint behaves in a standard learning task, both in terms of accuracy and in terms of the computational cost introduced by the inner solve.

\paragraph{Validation-based tuning protocol.}
To avoid test-set model selection, all hyperparameters are selected using a train/validation split. For each batch size
\[
b\in\{128,256,384\},
\]
we run a short validation-based grid search for each method under the same tuning budget. The hyperparameter maximizing validation accuracy at the end of the tuning run is selected. The test set is then used only once, in the final evaluation.

The grids are
\[
\eta_{\rm AdamW}\in
\{3\cdot10^{-4},5\cdot10^{-4},7\cdot10^{-4},10^{-3},1.5\cdot10^{-3},2\cdot10^{-3},3\cdot10^{-3}\},
\]
for AdamW,
\[
\alpha_{\rm NAG-GS}\in
\{0.2,0.35,0.5,0.75,1.0,1.5,2.0\},
\]
for NAG-GS, and
\[
\alpha_{\rm IRON}\in
\{0.75,1.0,1.25,1.5,2.0,2.5,3.0\},
\]
for \IRONfi{}.

\paragraph{Final evaluation.}
Using the selected hyperparameters, each method is run for \(25\) epochs and repeated over five random seeds. We report training loss, test accuracy, wall-clock accuracy curves, and final test accuracy with runtime. For \IRONfi{}, we also report the mean number of inner Newton--CG iterations, which measures the cost of the inexact resolvent evaluation. See Figures~\ref{fig:mnist-train-loss} to~\ref{fig:mnist-inner-iters}, and Table~\ref{tab:mnist_softmax_results}.

\begin{figure}[ht!]
  \centering

  \begin{subfigure}[t]{0.32\linewidth}
    \includegraphics[width=\linewidth]{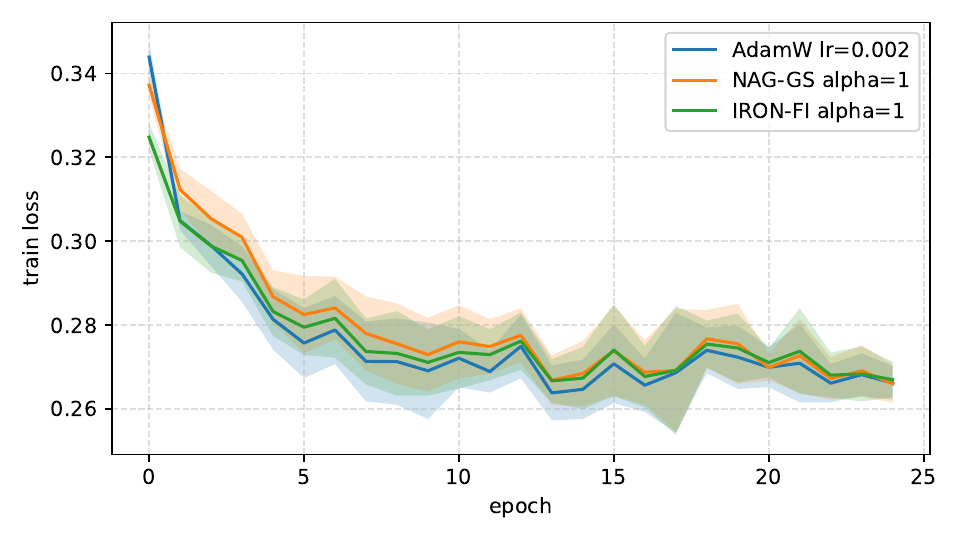}
    \caption{\(b=128\)}
  \end{subfigure}\hfill
  \begin{subfigure}[t]{0.32\linewidth}
    \includegraphics[width=\linewidth]{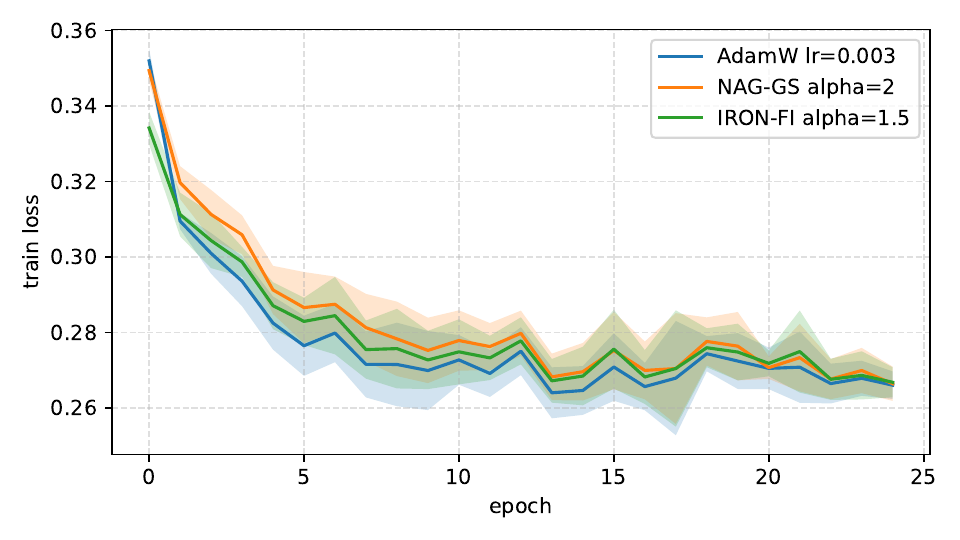}
    \caption{\(b=256\)}
  \end{subfigure}\hfill
  \begin{subfigure}[t]{0.32\linewidth}
    \includegraphics[width=\linewidth]{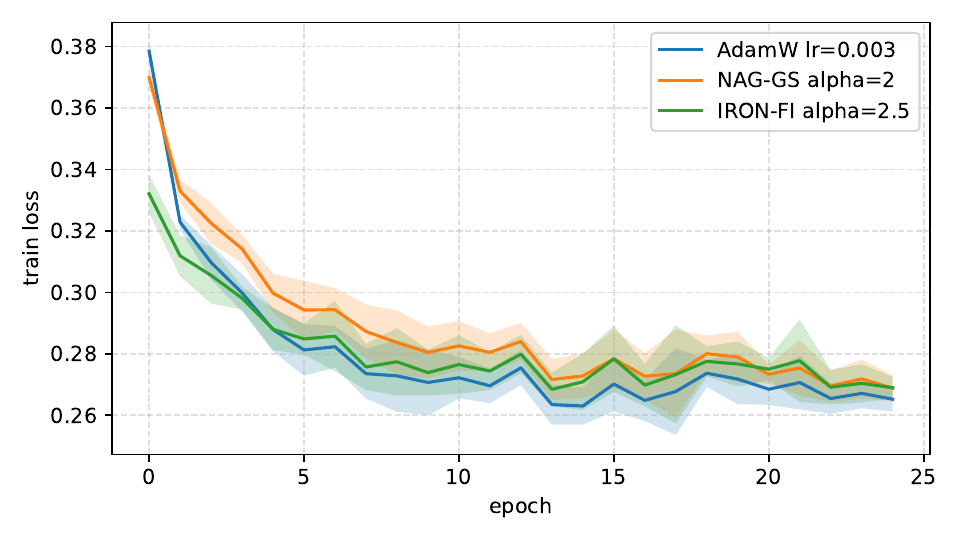}
    \caption{\(b=384\)}
  \end{subfigure}

  \caption{MNIST softmax regression: training loss versus epochs, averaged over five random seeds. \IRONfi{} is competitive in per-epoch loss decrease, especially for larger batch sizes where the validation-selected implicit stepsize is larger.}
  \label{fig:mnist-train-loss}
\end{figure}

\begin{figure}[ht!]
  \centering

  \begin{subfigure}[t]{0.32\linewidth}
    \includegraphics[width=\linewidth]{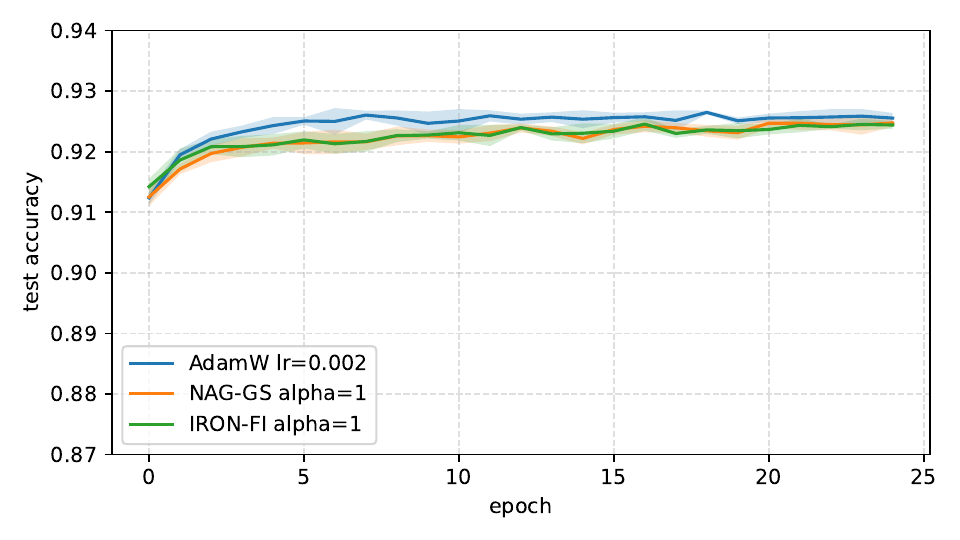}
    \caption{\(b=128\)}
  \end{subfigure}\hfill
  \begin{subfigure}[t]{0.32\linewidth}
    \includegraphics[width=\linewidth]{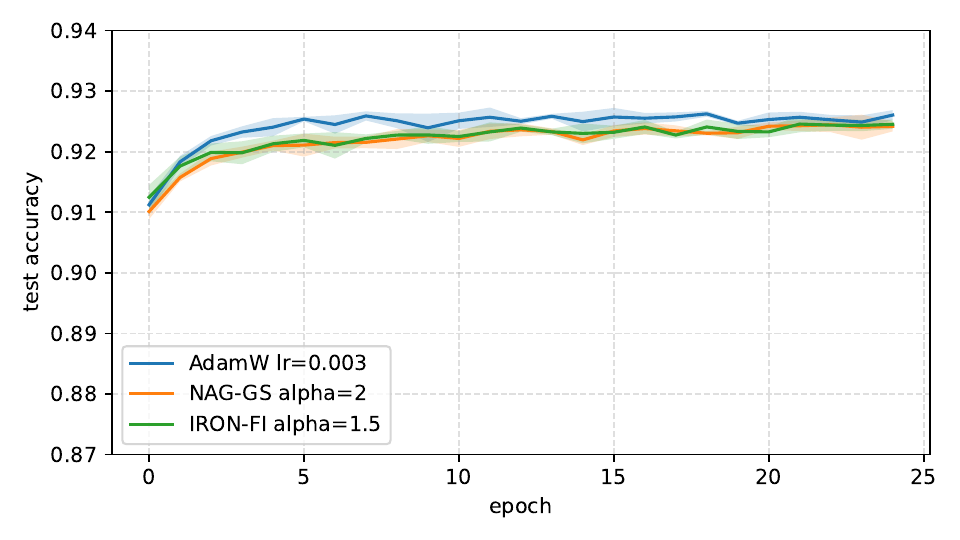}
    \caption{\(b=256\)}
  \end{subfigure}\hfill
  \begin{subfigure}[t]{0.32\linewidth}
    \includegraphics[width=\linewidth]{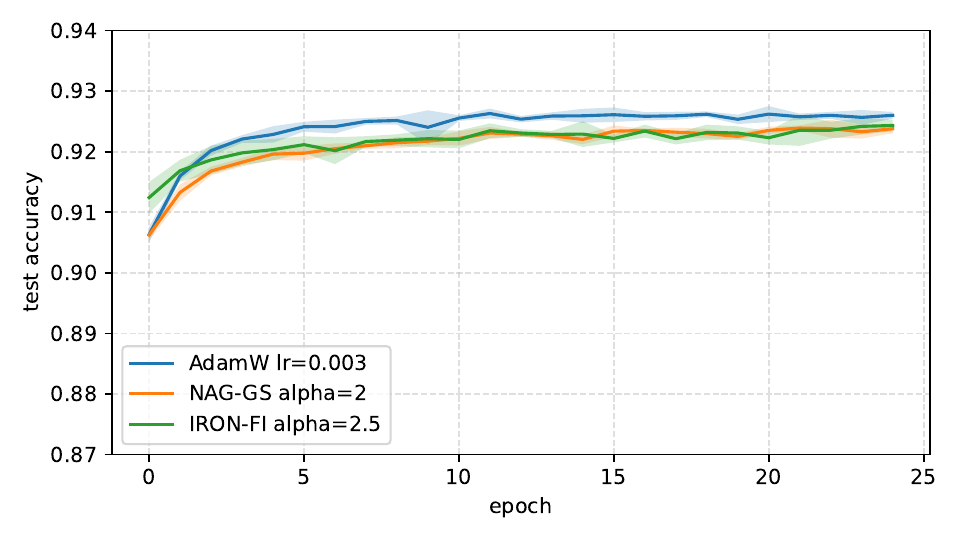}
    \caption{\(b=384\)}
  \end{subfigure}

  \caption{MNIST softmax regression: test accuracy versus epochs, averaged over five random seeds. With validation-selected parameters, \IRONfi{} remains close to AdamW and NAG-GS in final accuracy, while exhibiting stable trajectories across seeds.}
  \label{fig:mnist-test-acc-epoch}
\end{figure}

\begin{figure}[ht!]
  \centering

  \begin{subfigure}[t]{0.32\linewidth}
    \includegraphics[width=\linewidth]{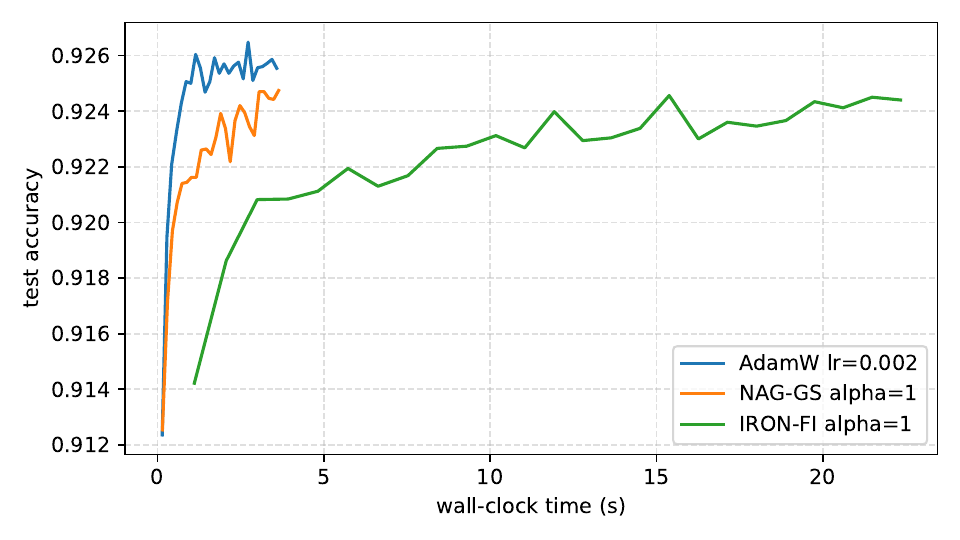}
    \caption{\(b=128\)}
  \end{subfigure}\hfill
  \begin{subfigure}[t]{0.32\linewidth}
    \includegraphics[width=\linewidth]{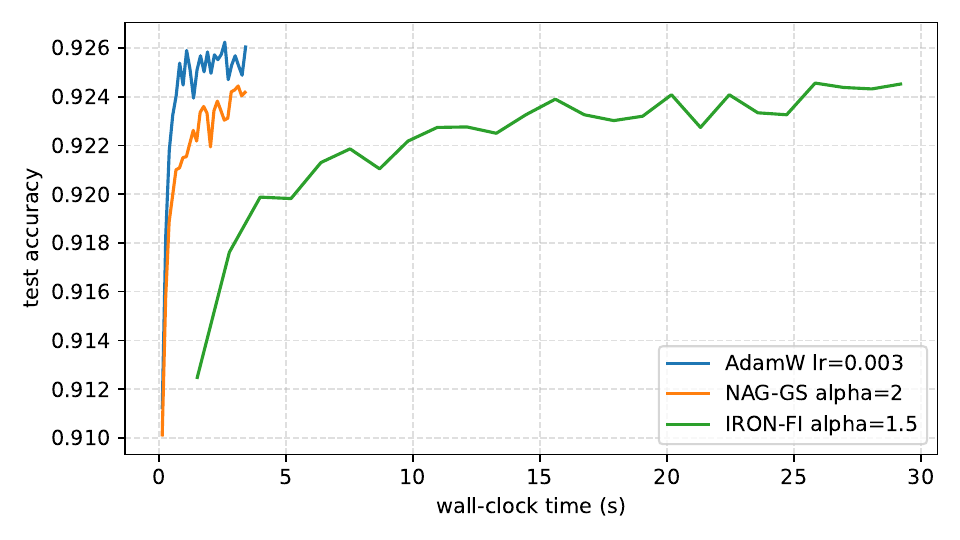}
    \caption{\(b=256\)}
  \end{subfigure}\hfill
  \begin{subfigure}[t]{0.32\linewidth}
    \includegraphics[width=\linewidth]{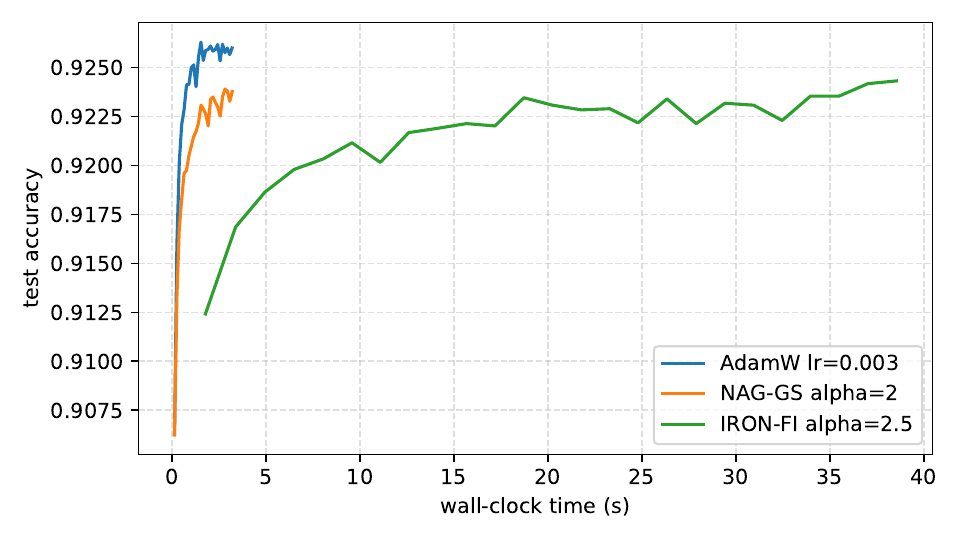}
    \caption{\(b=384\)}
  \end{subfigure}

  \caption{MNIST softmax regression: test accuracy versus wall-clock time. These plots complement the epoch-based curves and make the computational trade-off explicit. Although \IRONfi{} is competitive per epoch, its inexact Newton--CG inner solves make it slower in wall-clock time than the first-order baselines in this implementation.}
  \label{fig:mnist-test-acc-time}
\end{figure}

\begin{figure}[ht!]
  \centering

  \begin{subfigure}[t]{0.32\linewidth}
    \includegraphics[width=\linewidth]{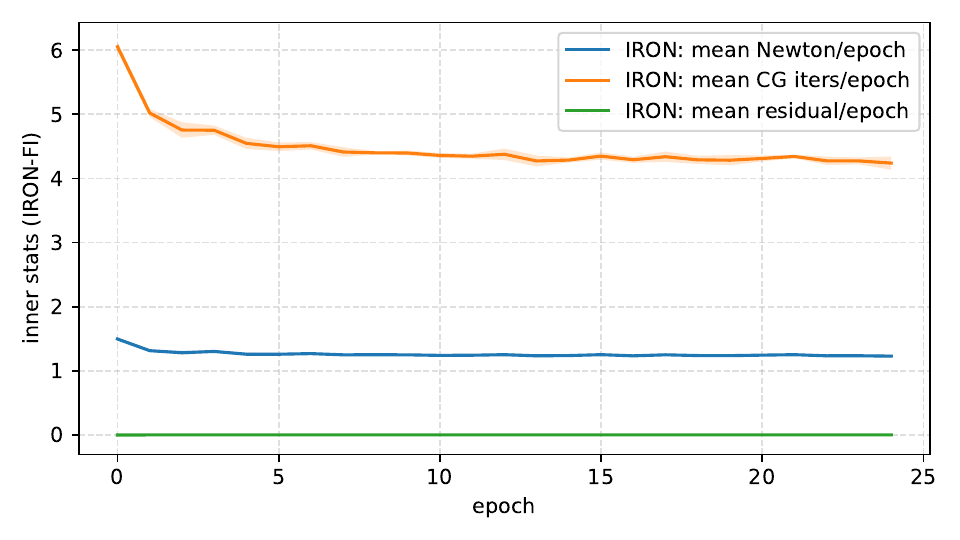}
    \caption{\(b=128\)}
  \end{subfigure}\hfill
  \begin{subfigure}[t]{0.32\linewidth}
    \includegraphics[width=\linewidth]{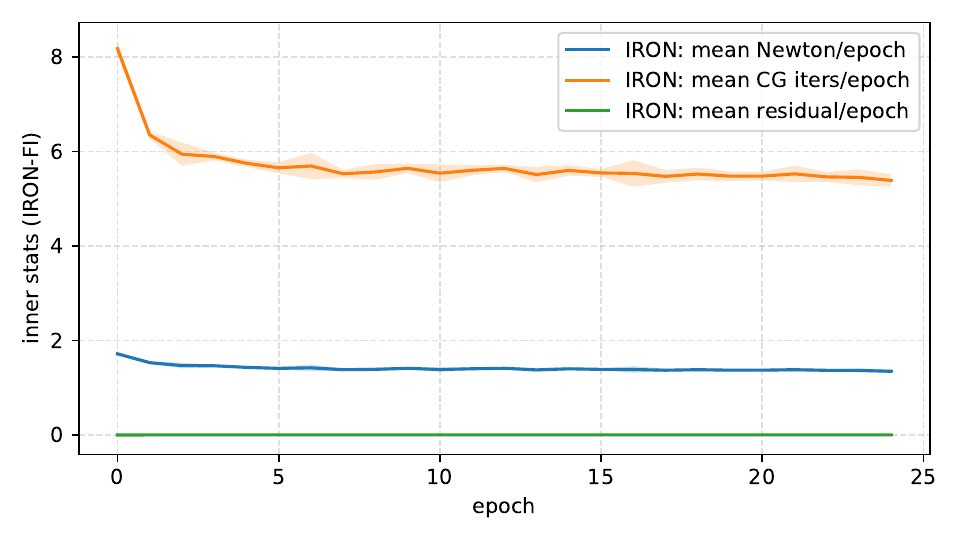}
    \caption{\(b=256\)}
  \end{subfigure}\hfill
  \begin{subfigure}[t]{0.32\linewidth}
    \includegraphics[width=\linewidth]{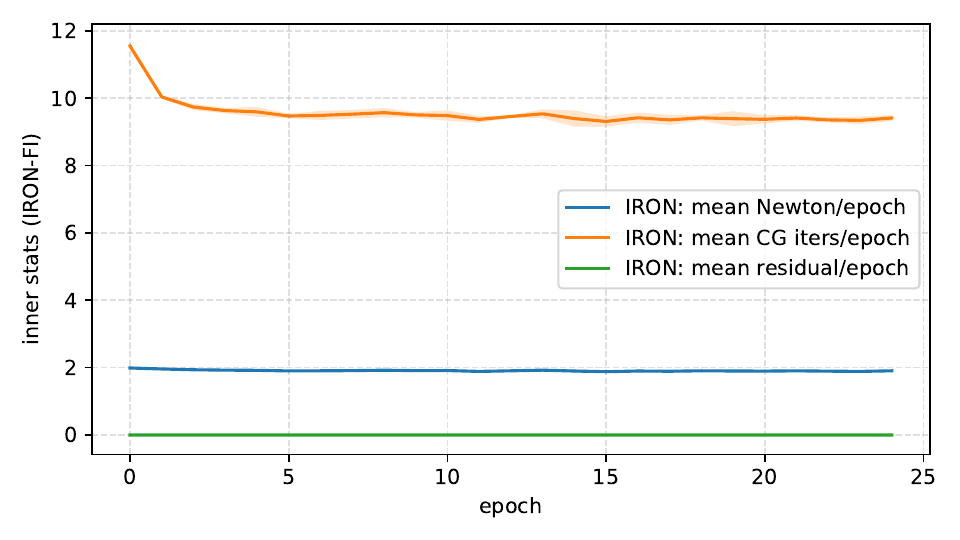}
    \caption{\(b=384\)}
  \end{subfigure}

  \caption{\IRONfi{} on MNIST: mean number of inner Newton--CG iterations per outer step. The inner iteration count quantifies the additional cost of evaluating the inexact resolvent and explains the wall-clock gap observed in Figure~\ref{fig:mnist-test-acc-time}.}
  \label{fig:mnist-inner-iters}
\end{figure}

% -------------------- Table --------------------

% -------------------- Table (cleaner + fused batch size) --------------------
\begin{table}[ht!]
\centering
\caption{MNIST softmax regression after 25 epochs, averaged over five random seeds. Hyperparameters are selected by the validation-based protocol described above and the test set is used only for the final evaluation. Reported values are mean\(\pm\)std over seeds. For each batch size, the best accuracy is shown in \best{bold}, and the second best is \second{underlined}.}
\label{tab:mnist_softmax_results}

\setlength{\tabcolsep}{5.5pt} % tighten columns a bit
\renewcommand{\arraystretch}{1.05}

\begin{tabular}{l l cc}
\toprule
Batch size $b$ & Method & Final test acc. @25 & Time (s) \\
\midrule
\multirow{3}{*}{128}
& AdamW ($\mathrm{lr}=0.002$)                           & \best{0.9255 $\pm$ 0.0009} & $3.58 \pm 0.06$ \\
& NAG-GS ($\mu=\gamma_0=1$, $\alpha=1.0$)               & \second{$0.9247 \pm 0.0006$} & $3.64 \pm 0.08$ \\
& IRON-FI ($\alpha=1.0$)                                & $0.9244 \pm 0.0006$ & $22.34 \pm 0.78$ \\
\midrule
\multirow{3}{*}{256}
& AdamW ($\mathrm{lr}=0.003$)                           & \best{0.9260 $\pm$ 0.0008} & $3.42 \pm 0.04$ \\
& NAG-GS ($\mu=\gamma_0=1$, $\alpha=2.0$)               & $0.9242 \pm 0.0008$ & $3.40 \pm 0.03$ \\
& IRON-FI ($\alpha=1.5$)                                & \second{$0.9245 \pm 0.0007$} & $29.19 \pm 0.24$ \\
\midrule
\multirow{3}{*}{384}
& AdamW ($\mathrm{lr}=0.003$)                           & \best{0.9260 $\pm$ 0.0006} & $3.19 \pm 0.06$ \\
& NAG-GS ($\mu=\gamma_0=1$, $\alpha=2.0$)               & $0.9238 \pm 0.0007$ & $3.20 \pm 0.07$ \\
& IRON-FI ($\alpha=2.5$)                                & \second{$0.9243 \pm 0.0009$} & $38.55 \pm 0.54$ \\

\bottomrule
\end{tabular}
\end{table}

\paragraph{Results.}
The validation-based protocol selects larger \IRONfi{} stepsizes as the batch size increases, consistent with the intuition that reduced mini-batch variability permits more aggressive implicit updates. In terms of epochs, \IRONfi{} is competitive with NAG-GS and remains within a small margin of AdamW in final test accuracy, while often decreasing the training loss rapidly during the first epochs. However, the wall-clock plots show the main practical limitation: the inexact Newton--CG inner solve dominates the runtime in the present implementation. Thus, these results should be read as evidence of stable behavior and competitive per-epoch progress, not as evidence of wall-clock superiority.

The inner-iteration plots further clarify this trade-off. They show the computational price paid for the resolvent evaluation and motivate faster inner solvers, structure-exploiting curvature approximations, randomized linear algebra, or cheaper inexact variants as important directions for future work.

\color{black}

\paragraph{Reproducibility and runtime.}
The full Python code, including scripts to regenerate all figures reported here, is available at
\url{https://github.com/vleplat/IRON_FI.git}.
The paper-facing MNIST runner performs validation-based tuning and final multi-seed evaluation. All experiments reported here were run on a MacBook Pro (M4 Pro, 24 GB memory); runtimes are wall-clock CPU measurements intended for relative comparison on this implementation. Absolute timings may differ substantially on GPU hardware or with optimized inner solvers.

\section{Conclusion and future work}\label{sec:conclusion}

We introduced \IRON, a discretization-driven framework for studying the long-time behavior of stochastic accelerated dynamics. Its fully implicit instantiation, \IRONfi{}, applies a Backward--Euler discretization to the accelerated stochastic flow and leads, after algebraic elimination, to a resolvent update. In this form, the stochastic perturbation injected in the velocity equation becomes an additive perturbation of the resolvent center, revealing a direct link between implicit discretization and stationary variance control.

For smooth strongly convex objectives, we proved that this mechanism yields a contractive Lyapunov recursion with a stationary mean-square error bound of order \(O(\sigma^2/\alpha)\). In the quadratic case, we further obtained an explicit asymptotic stationary constant and an exact covariance formula through a discrete Lyapunov equation. These results show that larger implicit steps can simultaneously strengthen contraction and reduce the stationary spread induced by noise.

%The experiments support this picture across controlled quadratic tests, synthetic strongly convex logistic regression, a qualitative nonconvex log-cosh example, and MNIST softmax regression. 
The controlled experiments support this picture in the quadratic test, the synthetic strongly convex logistic-regression benchmark, and the qualitative nonconvex log-cosh example. The MNIST softmax-regression benchmark complements these tests by illustrating the practical behavior
and computational cost of an inexact stochastic-oracle implementation of the resolvent viewpoint.
In the quadratic setting, the empirical stationary covariance matches the exact Lyapunov prediction. In the nonlinear strongly convex benchmark, the stationary mean-square error follows the predicted \(O(1/\alpha)\) scaling, with a fitted log--log slope close to \(-1\); the scaled quantity \(\alpha\,\widehat{\mathrm{MSE}}_\infty\) provides a complementary plateau diagnostic. On MNIST, \IRONfi{} exhibits stable accuracy trajectories and competitive per-epoch behavior, while highlighting the computational role of the inner Newton--CG/LM solve.

Several directions remain open. On the algorithmic side, the main priority is to reduce the cost of the inner resolvent solve, for example through warm starts, matrix-free methods, quasi-Newton or approximate curvature models, and randomized linear-algebra techniques. On the modeling side, it will be important to move beyond the Gaussian diffusion/center-noise proxy and better capture mini-batch sampling, non-isotropic covariance, and heavy-tailed perturbations. On the theory side, natural extensions include general convex objectives and local PL/K{\L} regimes around stable nonconvex stationary points. Finally, larger-scale implementations, including GPU-based solvers and deeper models, will help determine when the stationary-stability benefits of implicit discretization translate into end-to-end performance gains.

% ===================== Appendix with full proofs =====================
\appendix
\section{Appendix - Proofs}

\subsection{Proof of Lemma~\ref{lem:resolvent}}\label{app:proof-1}

\begin{proof}
Let $p=\prox_{\lambda f}(u)$ and $q=\prox_{\lambda f}(v)$. The optimality conditions give
$u=p+\lambda\nabla f(p)$ and $v=q+\lambda\nabla f(q)$, hence
$u-v=(p-q)+\lambda(\nabla f(p)-\nabla f(q))$.
Taking inner products with $p-q$ and using $\mu$-strong monotonicity of $\nabla f$ yields
$\langle u-v,p-q\rangle \ge (1+\lambda\mu)\|p-q\|^2$.
By Cauchy--Schwarz, $\langle u-v,p-q\rangle \le \|u-v\|\,\|p-q\|$, hence
$\|p-q\|\le (1+\lambda\mu)^{-1}\|u-v\|$.
\end{proof}

\subsection{Proof of Lemma~\ref{lem:vx}}\label{app:proof-2}
\begin{proof}
From \eqref{eq:v-gamma}, $v_{k}=x_{k}+\frac{x_k-x_{k-1}}{\alpha}$, so
\[
v_k-x^\star=(x_k-x^\star)+\frac{1}{\alpha}(x_k-x_{k-1}).
\]
By $(a+b)^2\le 2\|a\|^2+2\|b\|^2$ and $\alpha\ge 1$,

\begin{equation*}
    \begin{aligned}
        \|v_k-x^\star\|^2 & \le 2\|x_k-x^\star\|^2+\frac{2}{\alpha^2}\|x_k-x_{k-1}\|^2
\\
&\le 2\|x_k-x^\star\|^2+4\|x_k-x^\star\|^2+4\|x_{k-1}-x^\star\|^2,
    \end{aligned}
\end{equation*}
using $\|x_k-x_{k-1}\|^2\le 2\|x_k-x^\star\|^2+2\|x_{k-1}-x^\star\|^2$. This gives the stated $6$ and $4$ constants. The consequence for $\E\mathcal{E}_k$ follows from Lemma~\ref{lem:energy-eq}.
\end{proof}

\subsection{Proof of Lemma~\ref{lem:scale}}\label{app:proof-3}
\begin{proof}
Let $\tau=\frac{1}{\alpha}+\frac{\mu}{\gamma}$ and $\lambda=\frac{\alpha}{\gamma(1+\tau)}$ with $\alpha\ge 1$, $\gamma>0$, $\mu>0$.
First,

\begin{equation*}
    \begin{aligned}
        1+\lambda\mu \;=\; 1+\frac{\alpha\mu}{\gamma(1+\tau)}
\;& =\; 1+\frac{\alpha\mu}{\gamma\!\left(1+\frac{1}{\alpha}+\frac{\mu}{\gamma}\right)} \\
& \;\ge\; 1+\frac{\alpha\mu}{\gamma\!\left(2+\frac{\mu}{\gamma}\right)}\\
& \;\ge\; \frac{\alpha\mu}{\gamma\!\left(2+\frac{\mu}{\gamma}\right)}.
    \end{aligned}
\end{equation*}
Hence
\[
\frac{1}{(1+\lambda\mu)^2}
\;\le\; \Big(\frac{\gamma\!\left(2+\frac{\mu}{\gamma}\right)}{\mu}\Big)^2 \frac{1}{\alpha^2}
\;=:\; \frac{K_1}{\alpha^2}.
\]
Next, since $1+\tau=1+\frac{1}{\alpha}+\frac{\mu}{\gamma}\ge 1+\frac{\mu}{\gamma}$,
\[
\frac{1}{(1+\lambda\mu)^2}\cdot \frac{\alpha}{(1+\tau)^2}
\;\le\; \frac{K_1}{\alpha^2}\cdot \frac{\alpha}{(1+\mu/\gamma)^2}
\;=\; \frac{K_1}{(1+\mu/\gamma)^2}\cdot \frac{1}{\alpha}.
\]
This proves both claims. In particular, the product term scales as $\mathcal{O}(1/\alpha)$.
\end{proof}

\subsection{Proof of Proposition~\ref{prop:onestep}}\label{app:proof-5}
\begin{proof}
By Lemma~\ref{lem:resolvent} and \eqref{eq:proxsub},
\[
\norm{x_{k+1}-x^\star} \le \frac{1}{1+\lambda\mu}\,\norm{(c_k-x^\star)+\xi_k}.
\]
Square and take conditional expectation; the cross term vanishes since $\E[\xi_k\mid x_k,v_k]=0$:
\[
\E\!\left[\norm{x_{k+1}-x^\star}^2\,\middle|\,x_k,v_k\right]
\le \frac{1}{(1+\lambda\mu)^2}\Big(\norm{c_k-x^\star}^2+\E\norm{\xi_k}^2\Big).
\]
Use Lemma~\ref{lem:center} plus $\E\norm{\xi_k}^2\le \frac{\alpha}{(1+\tau)^2}\sigma^2$, and then Lemma~\ref{lem:scale} to get

\begin{equation*}
    \begin{aligned}
        \E\!\left[\norm{x_{k+1}-x^\star}^2\,\middle|\,x_k,v_k\right]
& \le \frac{K'}{\alpha^2}\big(\norm{x_k-x^\star}^2+\norm{v_k-x^\star}^2\big) \\
& \quad +\frac{\tilde K}{\alpha}\sigma^2.
    \end{aligned}
\end{equation*}
Finally apply Lemma~\ref{lem:energy-dom} and take expectations.
\end{proof}

\subsection{Proof of Theorem~\ref{thm:main}}\label{app:proof-6}

\begin{proof}
From Proposition~\ref{prop:onestep},
\[
\E\|x_{k+1}-x^\star\|^2
\;\le\; \frac{C_1}{\alpha^2}\,\E \mathcal{E}_k \;+\; \frac{C_2}{\alpha}\sigma^2.
\]
By Lemma~\ref{lem:vx} and Lemma~\ref{lem:energy-eq}, there exists $\bar m_2>0$ such that
\[
\E\mathcal{E}_k
\;\le\; \bar m_2\Big(\E\|x_k-x^\star\|^2 + \E\|x_{k-1}-x^\star\|^2\Big).
\]
Hence the two-step recursion holds:
\begin{equation}\label{eq:two-step-rec}
\E\|x_{k+1}-x^\star\|^2
\;\le\; \frac{A}{\alpha^2}\,\E\|x_k-x^\star\|^2
\;+\; \frac{B}{\alpha^2}\,\E\|x_{k-1}-x^\star\|^2
\;+\; \frac{C_2}{\alpha}\sigma^2,
\end{equation}
where $A,B>0$ depend only on $(\mu,L,\gamma)$.

Define $V_k := \E\|x_k-x^\star\|^2 + \theta\,\E\|x_{k-1}-x^\star\|^2$ with
$\theta := \sqrt{B}/\alpha$. Then by \eqref{eq:two-step-rec},
\[
V_{k+1}
\le \Big(\theta+\frac{A}{\alpha^2}\Big)\E\|x_k-x^\star\|^2
+ \frac{B}{\alpha^2}\E\|x_{k-1}-x^\star\|^2
+ \frac{C_2}{\alpha}\sigma^2.
\]
Set $\lambda := \frac{2\sqrt{B}}{\alpha} + \frac{A}{\alpha^2}$. We have
$\theta+\frac{A}{\alpha^2} \le \lambda$ and, since $\theta=\sqrt{B}/\alpha$ and $\alpha\ge 1$,
\[
\lambda\theta = \frac{2B}{\alpha^2}+\frac{A\sqrt{B}}{\alpha^3}\ge \frac{B}{\alpha^2}.
\]
Therefore
\[
V_{k+1}\le \lambda\Big(\E\|x_k-x^\star\|^2+\theta\E\|x_{k-1}-x^\star\|^2\Big)+\frac{C_2}{\alpha}\sigma^2
= \lambda V_k + \frac{C_2}{\alpha}\sigma^2.
\]
Finally, for $\alpha\ge 1$,
\[
\lambda \le \frac{2\sqrt{B}+A}{\alpha} =: \frac{G}{\alpha},
\]
which gives \eqref{eq:Vk-recursion} with $C:=C_2$. Iterating yields the closed form bound, and using $\E\|x_k-x^\star\|^2\le V_k$ concludes.
\end{proof}

\color{black}
\subsection{Proof of Proposition~\ref{prop:quad-explicit}}\label{app:quad-proof}

\begin{proof}
Fix an eigendirection \(i\), and write \(z_{k,i}:=(e_{k,i},w_{k,i})^\top\). From \eqref{eq:quad-2x2}, the recursion can be written as
\[
z_{k+1,i}=M_i z_{k,i}+g_i\xi_{k,i},
\]
where
\[
M_i=
\begin{pmatrix}
\dfrac{r_i\tau}{1+\tau} & \dfrac{r_i}{1+\tau}\\[2mm]
\Big(1+\dfrac1\alpha\Big)\dfrac{r_i\tau}{1+\tau}-\dfrac1\alpha
&
\Big(1+\dfrac1\alpha\Big)\dfrac{r_i}{1+\tau}
\end{pmatrix},
\qquad
g_i
=
r_i
\begin{pmatrix}
1\\[1mm]
1+\dfrac1\alpha
\end{pmatrix}.
\]
The \((2,1)\)-entry contains the additional term \(-1/\alpha\), but it is still \(O(1/\alpha)\). Indeed, since \(a_i\in[\mu,L]\) and
\[
r_i=\frac{1}{1+\lambda a_i}
=
\frac{1}{1+\frac{\alpha a_i}{\gamma(1+\tau)}}
=O(\alpha^{-1}),
\]
with constants depending only on \((\mu,L,\gamma)\), all entries of \(M_i\) are \(O(\alpha^{-1})\). Hence there exists \(K_i>0\), independent of \(\alpha\), such that
\[
\|M_i\|_2\le \frac{K_i}{\alpha},
\qquad
\|g_i\|_2\le \frac{K_i}{\alpha}.
\]
Using \(\E[\xi_{k,i}]=0\) and \(\E[\xi_{k,i}^2]=\alpha\rho^2/(1+\tau)^2=O(\alpha\rho^2)\), we obtain
\[
\E\!\left[\|z_{k+1,i}\|^2\,\middle|\,z_{k,i}\right]
=
\|M_i z_{k,i}\|^2+\E[\xi_{k,i}^2]\|g_i\|^2
\le
\frac{C_i}{\alpha^2}\|z_{k,i}\|^2+\frac{C_i\rho^2}{\alpha}.
\]
This proves Item~1.

We now prove the asymptotic stationary constant. Since \(\|M_i\|_2\le K_i/\alpha\), the recursion is stable for all sufficiently large \(\alpha\). Let \(P_i\) denote the stationary covariance of \(z_{k,i}\). Then
\[
P_i
=
\sum_{t=0}^{\infty} M_i^t Q_i(M_i^\top)^t,
\qquad
Q_i:=\E[\xi_{k,i}^2]\,g_i g_i^\top.
\]
Let \(p:=(1,0)^\top\). The stationary variance of the position coordinate is
\[
\E[e_{\infty,i}^2]
=
p^\top P_i p
=
\E[\xi_{k,i}^2]\sum_{t=0}^{\infty}\big(p^\top M_i^t g_i\big)^2.
\]
The leading term is the \(t=0\) term:
\[
\E[\xi_{k,i}^2]\big(p^\top g_i\big)^2
=
\E[\xi_{k,i}^2]\,r_i^2.
\]
The remaining terms are negligible at the scale \(1/\alpha\). Indeed, for \(t\ge 1\),
\[
|p^\top M_i^t g_i|
\le
\|M_i\|_2^t\|g_i\|_2
\le
\left(\frac{K_i}{\alpha}\right)^t\frac{K_i}{\alpha},
\]
and therefore
\[
\E[\xi_{k,i}^2]\sum_{t=1}^{\infty}\big(p^\top M_i^t g_i\big)^2
=
O(\alpha)\sum_{t=1}^{\infty}O(\alpha^{-2t-2})
=
O(\alpha^{-3}).
\]
Thus
\[
\E[e_{\infty,i}^2]
=
\E[\xi_{k,i}^2]\,r_i^2+O(\alpha^{-3}).
\]
Finally,
\[
1+\tau = 1+\frac1\alpha+\frac{\mu}{\gamma}
=
1+\frac{\mu}{\gamma}+o(1),
\]
and
\[
r_i
=
\frac{1}{1+\frac{\alpha a_i}{\gamma(1+\tau)}}
=
\frac{\gamma(1+\mu/\gamma)}{\alpha a_i}+o(\alpha^{-1}).
\]
Since
\[
\E[\xi_{k,i}^2]
=
\frac{\alpha\rho^2}{(1+\tau)^2}
=
\frac{\alpha\rho^2}{(1+\mu/\gamma)^2}+o(\alpha),
\]
we obtain
\[
\E[e_{\infty,i}^2]
=
\frac{\gamma^2\rho^2}{\alpha a_i^2}
+o(\alpha^{-1}).
\]
Multiplying by \(\alpha\) and summing over \(i\) gives
\[
\lim_{\alpha\to\infty}\alpha\,\E\|x_\infty-x^\star\|^2
=
\gamma^2\rho^2\sum_{i=1}^n\frac1{a_i^2}.
\]
Since \(\sigma^2=\tr(\Sigma)=n\rho^2\) in the isotropic case, this is equivalently
\[
\gamma^2\rho^2\sum_{i=1}^n\frac1{a_i^2}
=
\frac{\gamma^2}{n}\sigma^2\tr(A^{-2}).
\]
This proves Item~2.
\end{proof}

\color{black}

\color{black}
\subsection{Proof of Theorem~\ref{thm:vary}}\label{app:proof-varying-steps}

\begin{proof}
Define
\[
\gamma_-:=\min\{\gamma_0,\mu\},
\qquad
\gamma_+:=\max\{\gamma_0,\mu\}.
\]
The damping update \eqref{eq:v-gamma} is a convex combination of \(\gamma_k\) and \(\mu\):
\[
\gamma_{k+1}
=
\frac{1}{1+\alpha_k}\gamma_k+\frac{\alpha_k}{1+\alpha_k}\mu.
\]
Hence \(\gamma_k\in[\gamma_-,\gamma_+]\) for all \(k\).

We use the time-dependent energy
\[
\mathcal E_k^{\mathrm{var}}
:=
f(x_k)-f(x^\star)+\frac{\gamma_k}{2}\|v_k-x^\star\|^2.
\]
Uniformly in \(k\), strong convexity and smoothness give
\[
\frac{\mu}{2}\|x_k-x^\star\|^2
\le
\mathcal E_k^{\mathrm{var}}
\le
\frac{L}{2}\|x_k-x^\star\|^2+\frac{\gamma_+}{2}\|v_k-x^\star\|^2,
\]
and, conversely,
\[
\|x_k-x^\star\|^2+\|v_k-x^\star\|^2
\le
\left(\frac{2}{\mu}+\frac{2}{\gamma_-}\right)
\mathcal E_k^{\mathrm{var}}.
\]
The center-coupling estimate of Lemma~\ref{lem:center} and the resolvent-contraction argument in Proposition~\ref{prop:onestep} therefore hold with constants uniform over all
\[
\alpha_k\in[\underline{\alpha},\overline{\alpha}],
\qquad
\gamma_k\in[\gamma_-,\gamma_+].
\]
Moreover, the scaling estimate is uniform: there exist constants \(K,K'>0\), depending only on \((\mu,\gamma_-,\gamma_+,\underline{\alpha},\overline{\alpha})\), such that
\[
\frac{1}{(1+\lambda_k\mu)^2}\le \frac{K}{\alpha_k^2},
\qquad
\frac{1}{(1+\lambda_k\mu)^2}\frac{\alpha_k}{(1+\tau_k)^2}
\le
\frac{K'}{\alpha_k}.
\]
Thus, for some constants \(C_1,C_2>0\) independent of \(k\),
\[
\E\|x_{k+1}-x^\star\|^2
\le
\frac{C_1}{\alpha_k^2}\,\E\mathcal E_k^{\mathrm{var}}
+
\frac{C_2\sigma^2}{\alpha_k}.
\]

It remains to express the energy in terms of the position variables. From \eqref{eq:v-gamma},
\[
v_k=x_k+\frac{x_k-x_{k-1}}{\alpha_{k-1}},
\]
and since \(\alpha_{k-1}\ge \underline{\alpha}\ge 1\), the same argument as in Lemma~\ref{lem:vx} yields
\[
\|v_k-x^\star\|^2
\le
6\|x_k-x^\star\|^2+4\|x_{k-1}-x^\star\|^2.
\]
Consequently, there exist constants \(A,B,C>0\), independent of \(k\), such that
\[
u_{k+1}
\le
\frac{A}{\alpha_k^2}u_k
+
\frac{B}{\alpha_k^2}u_{k-1}
+
\frac{C\sigma^2}{\alpha_k},
\qquad
u_k:=\E\|x_k-x^\star\|^2.
\]
Since \(\alpha_k\ge \underline{\alpha}\), we also have
\[
u_{k+1}
\le
\frac{A}{\underline{\alpha}^{\,2}}u_k
+
\frac{B}{\underline{\alpha}^{\,2}}u_{k-1}
+
\frac{C\sigma^2}{\alpha_k}.
\]
Choose
\[
\theta:=\frac{\sqrt B}{\underline{\alpha}},
\qquad
V_k:=u_k+\theta u_{k-1}.
\]
As in the proof of Theorem~\ref{thm:main},
\[
V_{k+1}
\le
\left(\frac{2\sqrt B}{\underline{\alpha}}+\frac{A}{\underline{\alpha}^{\,2}}\right)V_k
+
\frac{C\sigma^2}{\alpha_k}.
\]
Since \(\underline{\alpha}\ge 1\), the prefactor is bounded by \(G/\underline{\alpha}\), with
\[
G:=2\sqrt B+A.
\]
This proves \eqref{eq:Vk-recursion-varying}. The stated nonasymptotic bound follows by unrolling the scalar recursion. If \(\underline{\alpha}>G\), then \(\rho=G/\underline{\alpha}<1\), and the limiting bound follows from the geometric series and the inequality \(\alpha_t\ge \underline{\alpha}\).
\end{proof}
\color{black}

% ---------- Bibliography ----------
% Fully-pathed BST: put spbasic.bst in springer/ (recommended)
\bibliographystyle{spmpsci}
\bibliography{aaai2026}

\begin{thebibliography}{10}
\providecommand{\url}[1]{{#1}}
\providecommand{\urlprefix}{URL }
\expandafter\ifx\csname urlstyle\endcsname\relax
  \providecommand{\doi}[1]{DOI~\discretionary{}{}{}#1}\else
  \providecommand{\doi}{DOI~\discretionary{}{}{}\begingroup \urlstyle{rm}\Url}\fi

\bibitem{asi2020aprox}
Asi, H., Chadha, K., Cheng, G., Duchi, J.C.: Minibatch stochastic approximate proximal point methods.
\newblock In: Advances in Neural Information Processing Systems, vol.~33 (2020).
\newblock \urlprefix\url{https://proceedings.neurips.cc/paper/2020/hash/fa2246fa0fdf0d3e270c86767b77ba1b-Abstract.html}

\bibitem{bauschke2017convex}
Bauschke, H.H., Combettes, P.L.: Convex Analysis and Monotone Operator Theory in Hilbert Spaces, 2 edn.
\newblock Springer (2017).
\newblock \doi{10.1007/978-3-319-48311-5}

\bibitem{combettes2015sfb}
Combettes, P.L., Pesquet, J.: Stochastic approximations and perturbations of forward--backward splitting methods.
\newblock arXiv preprint arXiv:1507.07095  (2015)

\bibitem{conn2000trust}
Conn, A.R., Gould, N.I.M., Toint, P.L.: Trust-Region Methods.
\newblock SIAM, Philadelphia (2000).
\newblock \doi{10.1137/1.9780898719857}

\bibitem{dahlquist1963special}
Dahlquist, G.G.: A special stability problem for linear multistep methods.
\newblock BIT Numerical Mathematics \textbf{3}(1), 27--43 (1963).
\newblock \doi{10.1007/BF01963532}

\bibitem{davis2019modelbased}
Davis, D., Drusvyatskiy, D.: Stochastic model-based minimization of weakly convex functions.
\newblock SIAM Journal on Optimization \textbf{29}(1), 207--239 (2019)

\bibitem{duchi2009forwardbackward}
Duchi, J., Singer, Y.: Efficient online and batch learning using forward backward splitting.
\newblock Journal of Machine Learning Research \textbf{10}(99), 2899--2934 (2009).
\newblock \urlprefix\url{https://jmlr.org/papers/v10/duchi09a.html}

\bibitem{hairer1996sode2}
Hairer, E., Wanner, G.: Solving Ordinary Differential Equations II: Stiff and Differential-Algebraic Problems, 2 edn.
\newblock Springer, Berlin (1996).
\newblock \doi{10.1007/978-3-642-05221-7}

\bibitem{kim2022sppam}
Kim, J.L., Toulis, P., Kyrillidis, A.: Convergence and stability of the stochastic proximal point algorithm with momentum.
\newblock In: Proceedings of The 4th Annual Learning for Dynamics and Control Conference, \emph{Proceedings of Machine Learning Research}, vol. 168, pp. 1034--1047. PMLR (2022).
\newblock \urlprefix\url{https://proceedings.mlr.press/v168/kim22a.html}

\bibitem{kingma2015adam}
Kingma, D.P., Ba, J.: Adam: A method for stochastic optimization.
\newblock In: International Conference on Learning Representations (ICLR) (2015).
\newblock \urlprefix\url{https://arxiv.org/abs/1412.6980}

\bibitem{leplat2023naggssemiimplicitacceleratedrobust}
Leplat, V., Merkulov, D., Katrutsa, A., Bershatsky, D., Tsymboi, O., Oseledets, I.: Nag-gs: Semi-implicit, accelerated and robust stochastic optimizer (2023).
\newblock \urlprefix\url{https://arxiv.org/abs/2209.14937}

\bibitem{levenberg1944method}
Levenberg, K.: A method for the solution of certain non-linear problems in least squares.
\newblock Quarterly of Applied Mathematics \textbf{2}(2), 164--168 (1944)

\bibitem{loshchilov2019decoupled}
Loshchilov, I., Hutter, F.: Decoupled weight decay regularization.
\newblock In: International Conference on Learning Representations (ICLR) (2019).
\newblock \urlprefix\url{https://arxiv.org/abs/1711.05101}

\bibitem{luo2022from}
Luo, H., Chen, L.: From differential equation solvers to accelerated first-order methods for convex optimization.
\newblock Mathematical Programming \textbf{195}(1), 735--781 (2022).
\newblock \doi{10.1007/s10107-021-01713-3}

\bibitem{marquardt1963algorithm}
Marquardt, D.W.: An algorithm for least-squares estimation of nonlinear parameters.
\newblock Journal of the Society for Industrial and Applied Mathematics \textbf{11}(2), 431--441 (1963).
\newblock \doi{10.1137/0111030}

\bibitem{milzarek2024spp}
Milzarek, A., Schaipp, F., Ulbrich, M.: A semismooth newton stochastic proximal point algorithm with variance reduction.
\newblock SIAM Journal on Optimization \textbf{34}(1), 1157--1185 (2024).
\newblock \doi{10.1137/22M1488181}

\bibitem{patrascu2018spp}
P{\u a}tra{\c s}cu, A., Necoara, I.: Nonasymptotic convergence of stochastic proximal point methods for constrained convex optimization.
\newblock Journal of Machine Learning Research \textbf{18}(198), 1--42 (2018)

\bibitem{richtarik2024unified}
Richt{\'a}rik, P., Sadiev, A., Demidovich, Y.: A unified theory of stochastic proximal point methods without smoothness (2024).
\newblock \urlprefix\url{https://arxiv.org/abs/2405.15941}

\bibitem{toulis2017implicit}
Toulis, P., Airoldi, E.M.: Asymptotic and finite-sample properties of estimators based on stochastic gradients.
\newblock The Annals of Statistics \textbf{45}(4), 1694--1727 (2017).
\newblock \doi{10.1214/16-AOS1506}

\bibitem{toulis2021proxrm}
Toulis, P., Horel, T., Airoldi, E.M.: The proximal robbins--monro method.
\newblock Journal of the Royal Statistical Society: Series B \textbf{83}(1), 188--212 (2021).
\newblock \doi{10.1111/rssb.12405}

\bibitem{traore2024variance}
Traor{\'e}, C., Apidopoulos, V., Salzo, S., Villa, S.: Variance reduction techniques for stochastic proximal point algorithms.
\newblock Journal of Optimization Theory and Applications \textbf{203}, 1910--1939 (2024).
\newblock \doi{10.1007/s10957-024-02502-6}

\bibitem{wilson2021lyapunov}
Wilson, A.C., Recht, B., Jordan, M.I.: A lyapunov analysis of accelerated methods in optimization.
\newblock Journal of Machine Learning Research \textbf{22}(202), 1--34 (2021).
\newblock \urlprefix\url{https://jmlr.org/papers/v22/20-195.html}

\bibitem{xiao2014proxsvrg}
Xiao, L., Zhang, T.: A proximal stochastic gradient method with progressive variance reduction.
\newblock SIAM Journal on Optimization \textbf{24}(4), 2057--2075 (2014).
\newblock \doi{10.1137/140961791}

\end{thebibliography}

\end{document}